\documentclass[11pt]{amsart}
\usepackage{amsmath,amsthm,amssymb,amsbsy,amstext,amsopn}
\usepackage{aliascnt}
\usepackage{xcolor}
\usepackage{graphicx}
\usepackage{tikz-cd}
\usepackage{microtype}
\usepackage[centering,margin=1.2in]{geometry}
\usepackage[colorlinks,
            linkcolor=black!50!red,
            citecolor=blue,
            pdfpagemode=None]{hyperref}

\numberwithin{equation}{section}

\theoremstyle{plain}

\newaliascnt{theorem}{equation}
\newtheorem{thm}[theorem]{Theorem}
\aliascntresetthe{theorem}

\newaliascnt{prop}{equation}
\newtheorem{prop}[prop]{Proposition}
\aliascntresetthe{prop}

\newaliascnt{lemma}{equation}
\newtheorem{lemma}[lemma]{Lemma}
\aliascntresetthe{lemma}

\newaliascnt{corollary}{equation}
\newtheorem{corollary}[corollary]{Corollary}
\aliascntresetthe{corollary}

\theoremstyle{definition}

\newaliascnt{remark}{equation}
\newtheorem{remark}[remark]{Remark}
\aliascntresetthe{remark}

\newaliascnt{defn}{equation}
\newtheorem{defn}[defn]{Definition}
\aliascntresetthe{defn}

\newaliascnt{example}{equation}
\newtheorem{example}[example]{Example}
\aliascntresetthe{example}

\usepackage{cleveref}

\crefname{theorem}{Theorem}{Theorems}
\crefname{prop}{Proposition}{Propositions}
\crefname{section}{Section}{Sections}
\crefname{subsection}{Section}{Sections}
\crefname{defn}{Definition}{Definitions}
\crefname{remark}{Remark}{Remarks}
\crefname{lemma}{Lemma}{Lemmas}
\crefname{corollary}{Corollary}{Corollaries}
\crefname{example}{Example}{Examples}

\newcommand{\mb}[1]{\mathbf{#1}}

\DeclareMathOperator{\Ad}{Ad}

\DeclareMathOperator{\Hom}{Hom}
\DeclareMathOperator{\End}{End}

\DeclareMathOperator{\Spec}{Spec}

\newcommand{\wt}{\widetilde}

\newcommand{\op}{\mathrm{op}}
\newcommand{\ot}{\otimes}
\newcommand{\nin}{\notin}
\newcommand{\comments}[1]{}
\newcommand{\ol}[1]{\overline{#1}}
\newcommand{\dol}[1]{\overline{\overline{#1}}}

\def\al{\alpha}
\def\be{\beta}
\def\de{\delta}
\def\De{\Delta}

\def\io{\iota}

\def\la{\lambda}
\def\th{\theta}
\def\si{\sigma}
\def\Si{\Sigma}

\def\om{\omega}

\def\ep{\epsilon}

\def\ze{\zeta}

\def\CA{{\mathcal A}}
\def\CB{{\mathcal B}}

\def\CX{{\mathcal X}}

\def\BA{{\mathbb A}}

\def\BC{{\mathbb C}}

\def\BN{{\mathbb N}}

\def\BP{{\mathbb P}}
\def\BQ{{\mathbb Q}}
\def\BR{{\mathbb R}}

\def\BZ{{\mathbb Z}}

\newcommand{\fg}{\mathfrak g}
\def\hh{{\mathfrak h}}

\begin{document}

\title{Cluster Ensembles and Kac-Moody Groups}
\author{Harold Williams}
\address{Harold Williams\newline
University of California, Berkeley\newline
Department of Mathematics\newline
Berkeley CA 94720\newline
USA}
\email{harold@math.berkeley.edu}

%\address{Department of Mathematics, University of California at Berkeley, Berkeley CA 94720, USA}
%\ead{harold@math.berkeley.edu}

\begin{abstract}
We study the relationship between two sets of coordinates on a double Bruhat cell, the cluster variables introduced by Berenstein, Fomin, and Zelevinsky and the $\CX$-coordinates defined by the coweight parametrization of Fock and Goncharov.  In these coordinates, we show that the generalized Chamber Ansatz of Fomin and Zelevinsky is a nondegenerate version of the canonical monomial transformation between the cluster variables and $\CX$-coordinates defined by a common exchange matrix.  We prove this in the setting of an arbitrary symmetrizable Kac-Moody group, generalizing along the way many previous results on the double Bruhat cells of a semisimple algebraic group.  In particular, we construct an upper cluster algebra structure on the coordinate ring of any double Bruhat cell in a symmetrizable Kac-Moody group, proving a conjecture of Berenstein, Fomin, and Zelevinsky. 
\end{abstract}

\maketitle

\section{Introduction}

Cluster algebras were discovered by Fomin and Zelevinsky in the context of dual canonical bases and total positivity in semisimple algebraic groups.  Their formulation was based in part on identities satisfied by generalized minors encountered in the study of double Bruhat cells \cite{Fomin1999}.  These minors were used to write explicit formulas for the inverses of certain birational parametrizations of these cells, generalizing the Chamber Ansatz previously introduced in the context of unipotent cells \cite{Berenstein1996,Berenstein1997}.  After the axiomatization of cluster algebras in \cite{Fomin2001}, these generalized minors were reinterpreted as cluster variables in an upper cluster algebra structure on the coordinate ring of the double Bruhat cell \cite{Berenstein2005}.  In the present article we extend this family of results to the more general setting of symmetrizable Kac-Moody groups, and in particular construct the corresponding cluster algebras.  Furthermore, we show that the generalized Chamber Ansatz is a component of a larger cluster ensemble formed by the simply-connected and adjoint forms of the double Bruhat cell.

The structure of a cluster algebra is encoded in the combinatorial datum of an exchange matrix.  From such a matrix others may be produced by an iterative process of mutation, and the cluster algebra is determined by the collection of all matrices obtained in this way.  Soon after \cite{Fomin2001} it was discovered that the dynamics of mutations encode a second type of algebraic structure, variously called coefficients or $Y$-variables \cite{Fomin2006}, $\tau$-coordinates \cite{Gekhtman2002}, and $\CX$-coordinates \cite{Fock2003}.  In \cite{Fock2006} a class of such coordinates were constructed on the double Bruhat cells of the adjoint form of a semisimple algebraic group.  These are given by another family of birational parametrizations of the cell, related to those studied in \cite{Fomin1999} but defined in terms of coweight subgroups rather than one-parameter unipotent subgroups.  However, the relationship between these $\CX$-coordinates and the cluster variables of \cite{Berenstein2005} was not studied explicitly.

In general, the cluster algebra and $\CX$-coordinates encoded by a common exchange matrix are related by a canonical map, defined abstractly as a Laurent monomial transformation whose exponents are the entries of the exchange matrix.   Concrete instances of this include the projection from decorated Teichm\"{u}ller space to Teichm\"{u}ller space \cite{Fock2005} and the transformation of $T$-system solutions to corresponding $Y$-system solutions \cite{Kuniba2011}.  It was first defined in terms of exchange matrices in the study of compatible Poisson structures on a cluster algebra \cite{Gekhtman2002}, and in \cite{Fomin2006} played a key role in the derivation of universal formulas for cluster variables in terms of $F$-polynomials.  Following the terminology of \cite{Fock2003} we refer to it as the cluster ensemble map; one of our main results is that the generalized Chamber Ansatz of \cite{Fomin1999}, when expressed in terms of the coweight parametrization of a double Bruhat cell, is a certain nondegenerate version of this structure (see \cref{thm:ensthm}).  In particular, this change of variables turns the initially opaque formulas of \cite{Fomin1999} into ones whose form is completely intuitive from the perspective of the general theory.

Our broader goal is to extend the constructions of \cite{Fomin1999,Berenstein2005,Fock2006} to the setting of arbitrary symmetrizable Kac-Moody groups.  These groups share many structural properties with semisimple algebraic groups, in particular a decomposition into finite-dimensional double Bruhat cells.  We show that the coordinate rings of all such double Bruhat cells are upper cluster algebras, verifying a conjecture of \cite{Berenstein2005} (see \cref{thm:clusterthm}).  For the unipotent cells of a Kac-Moody group, this cluster algebra structure was described in \cite{Geiss2010,Demonet2011}.  Their treatment is based on the representation theory of preprojective algebras, and provides categorical interpretations of many earlier group-theoretic constructions \cite{Geiss2012}.  However, at the present time this framework does not extend beyond unipotent cells; instead, we adapt the techniques of \cite{Fomin1999,Zelevinsky2000,Berenstein2005} to the infinite-dimensional case.

Whereas cluster variables are motivated by the theory of canonical bases, $\CX$-coordinates are more natural from the perspective of Poisson geometry.  In particular, an exchange matrix endows the corresponding $\CX$-coordinates with a canonical Poisson bracket, which in the case of double Bruhat cells coincides with the Sklyanin bracket.  The characters of the group restrict to Poisson-commuting functions on the double Bruhat cell, and in some cases form a completely integrable system \cite{Hoffmann2000,Reshetikhin2003}.  Many interesting examples come from non-unipotent cells in affine Kac-Moody groups, and this is one of our main motivations for studying double Bruhat cells in this generality \cite{Williams2012,Marshakov2012}.  Moreover, this context calls specific attention to role of the coweight parametrization, in that the resulting $\CX$-coordinates provide the link between these systems and those constructed from the dimer partition function of a bipartite torus graph \cite{Fock2012,Goncharov2011}.  

The layout of the paper is as follows. In \cref{sec:kmgroups} we recall the necessary background on Kac-Moody groups and discuss their generalized minors.  In \cref{sec:dbc} we study various coordinate systems on the double Bruhat cells of such groups.  In particular, we generalize the Chamber Ansatz of \cite{Fomin1999} to the Kac-Moody case, and derive the analogous formula for the coweight parametrization of \cite{Fock2006}.  From the latter we recover the exchange matrix defined in \cite{Berenstein2005}, and in \cref{sec:clusterdbc} we consider the corresponding cluster structures associated with the double Bruhat cell, summarizing the main results in \cref{thm:ensthm}.

\textsc{Acknowledgements}  I would like to generously thank Bernard Leclerc, Vladimir Fock, Nicolai Reshetikhin, Lauren Williams, Pablo Solis, and Qi You for valuable discussions and comments.  A particular gratitude is owed to the late Andrei Zelevinsky for a number of useful and encouraging remarks.  This research was supported by NSF grant DMS-0943745 and the Centre for Quantum Geometry of Moduli Spaces at Aarhus University.

\section{Kac-Moody Groups and Generalized Minors}\label{sec:kmgroups}

\subsection{Kac-Moody Algebras}\label{sec:KMalg} 
We briefly recall the theory of Kac-Moody algebras \cite{Kac1994}.  A generalized Cartan matrix $C$ is an $r \times r$ integer matrix such that
\begin{enumerate}
\item $C_{ii} = 2$ for all $1 \leq i \leq r$
\item $C_{ij} \leq 0$ for $i \neq j$
\item $C_{ij} = 0$ if and only if $C_{ji} = 0$.
\end{enumerate}
We will assume throughout that $C$ is symmetrizable; that is, there exist positive integers $d_1,\dots,d_r$ such that $d_i C_{ij} = d_j C_{ji}$ for all $1 \leq i,j \leq r$.  To the matrix $C$ is associated a Lie algebra $\fg:=\fg(C)$.  The Cartan subalgebra $\hh \subset \fg$ contains simple coroots $\{ \al^\vee_1, \dots , \al^\vee_r\}$, its dual contains simple roots $\{ \al_1, \dots, \al_r\}$, and these satisfy $\langle \al_j | \al_i^\vee \rangle = C_{ij}$.  The dimension of $\hh$, which we denote throughout by $\wt{r}$, is equal to $2r - \text{rank}(C)$.

The algebra $\fg$ is generated by $\hh$ and the Chevalley generators $\{ e_1, f_1, \dots, e_r, f_r\}$, subject to the relations
\begin{enumerate}
\item $[h,h'] = 0$ for all $h, h' \in \hh$
\item $[h,e_i] = \langle \al_i | h \rangle e_i$
\item $[h,f_i] = - \langle \al_i | h \rangle f_i$
\item $[e_i,f_i] = \al^\vee_i$
\item $[e_i,f_j] = ad(e_i)^{1-C_{ij}}e_j = ad(f_i)^{1-C_{ij}}f_j = 0$ for all $i \neq j$.
\end{enumerate}

The roots of $\fg$ are the elements $\al \in \hh^*$ such that
\[
\fg_{\al} =  \{X \in \fg\: |\: [h,X] = \langle \al|h \rangle X \mathrm{\:\: for\: all\:\:} h \in \hh \}
\]
is nonzero.  Any nonzero root is a sum of simple roots with either all positive or all negative integer coefficients, and we say it is a positive or negative root accordingly.  

The Weyl group $W$ is the subgroup of $\mathrm{Aut}(\hh^*)$ generated by the simple reflections
\[
s_i: \be \mapsto \be - \langle \be|\al^\vee_i \rangle \al_i.
\]
 A nonzero root is said to be real if it is conjugate to a simple root under $W$, and imaginary otherwise.  A reduced word for an element of $W$ is an expression $w = s_{i_1} \cdots s_{i_n}$ such that $n$ is as small as possible; the length $\ell(w)$ is then defined as the length of such a reduced word. 

We fix a complex algebraic torus $H$ with Lie algebra $\hh$, which in the following section will be the Cartan subgroup of the group associated with $\fg$.  The integral weight lattice $P := \mathrm{Hom}(H,\BC^*)$ can be regarded as a sublattice of $\hh^*$, with
\[
\langle \om | \al^\vee_i \rangle \in \BZ
\]
for all $\om \in P$ and all simple coroots $\al^\vee_i$.  We fix once and for all a basis $\{ \om_1,\dots, \om_{\wt{r}} \}$ of $P$, the \emph{fundamental weights}, such that
\[
\langle \om_j | \al^\vee_i \rangle = \de_{i,j}, \quad 1\leq i \leq r, \quad 1 \leq j \leq \wt{r}.
\]
The choice of fundamental weights lets us uniquely define $C_{ij}$ for $r \leq i \leq \wt{r}$ by the requirement that
\begin{equation}\label{eq:extcartan}
\al_j = \sum_{1 \leq i \leq \wt{r}} C_{ij}\om_i.
\end{equation}

Given $a \in H$, we will denote the value of the character $\la \in P$ at $a$ as $a^\la$.  Conversely, given $t \in \BC^*$ and a cocharacter $\la^\vee \in \Hom(\BC^*,H)$, we write $t^{\la^\vee}$ for the corresponding element of $H$.  Having fixed the basis $\om_1,\dots,\om_{\wt{r}}$ of $P$, we have a corresponding dual basis of the cocharacter lattice $\Hom(\BC^*,H)$.  We denote its elements by $\al_1^\vee,\dots,\al_{\wt{r}}^\vee$, since for $i<r$ these are just the coroots of $G$.   

The set of dominant weights is $P_+ := \{\la \in P : \langle \la | \al^\vee_i \rangle \geq 0 \text{ for all } 1 \leq i \leq r\}$.  For each $\la \in P_+$ there is an irreducible $\fg$-representation $L(\la)$ with highest weight $\la$, unique up to isomorphism.  The representation $L(\la)$ is the direct sum of finite-dimensional $\hh$-weight spaces, and its graded dual $L(\la)^\vee$ is an irreducible lowest-weight representation. 

Let $\si$ be the involution of $\fg$ determined by
\begin{equation}\label{eq:si}
\si(h) = -h \text{ for all } h \in H, \quad \si(e_i) = -f_i, \quad \si(f_i) = -e_i,
\end{equation}
and let $\rho_\la: \fg \to \End{L(\la)}$ be the map defining the action of $\fg$ on $L(\la)$.   Then there is a $\fg$-module isomorphism between $L(\la)^\vee$ and the representation whose underlying vector space is $L(\la)$ and whose $\fg$-action is given by $\rho_\la \circ \si$.  In particular this isomorphism yields a nondegenerate symmetric bilinear form
\[
L(\la)\ot L(\la) \cong L(\la)^\vee \ot L(\la) \to \BC.
\]

\subsection{Kac-Moody Groups and Double Bruhat Cells}\label{subsec:KMgroups}

To a generalized Cartan matrix $C$ we may also associate a group $G$, which is a simply-connected complex algebraic group when $C$ is positive-definite \cite{Peterson1983,Kumar2002}.  In general $G$ is an ind-algebraic group, and shares many important properties with the simple algebraic groups, in particular a Bruhat decomposition and generalized Gaussian factorization.

For each real root $\al$, $G$ contains a one-parameter subgroup $x_{\al}(t)$, and $G$ is generated by these together with the Cartan subgroup $H$ (for simple roots, we will write $x_{\pm i}(t) := x_{\pm \al_i}(t)$).  We denote the subgroups generated by the positive and negative real root subgroups by $N_+$ and $N_-$, respectively, and we also have the positive and negative Borel subgroups $B_{\pm} := H \ltimes N_{\pm}$.  

For each $1 \leq i \leq r$ there is a unique embedding $\varphi_i : SL_2 \to G$ such that
\[
\varphi_i \begin{pmatrix} t & 0 \\ 0 & t^{-1} \end{pmatrix} = t^{\al^\vee_i}, \quad  \varphi_i \begin{pmatrix} 1 & t \\ 0 & 1 \end{pmatrix} = x_i(t), \quad  \varphi_i \begin{pmatrix} 1 & 0 \\ t & 1 \end{pmatrix} = x_{-i}(t).
\]
The Weyl group $W$ is isomorphic with $N_G(H)/H$, where $N_G(H)$ is the normalizer of $H$ in $G$.  The simple reflections $s_{i}$ have representatives in $G$ of the form
\begin{align}\label{eqn:simplerootformula}
\ol{s_{i}} = x_{i}(-1) x_{-i}(1) x_{i}(-1) = \varphi_i \begin{pmatrix} 0 & -1 \\ 1 & 0 \end{pmatrix} \\
\dol{s_{i}} = x_{i}(1) x_{-i}(-1) x_{i}(1) = \varphi_i \begin{pmatrix} 0 & 1 \\ -1 & 0 \end{pmatrix}.
\end{align}  
In particular, for any $w \in W$ we have well-defined representatives
\[
\ol{w} = \ol{s_{i_1}} \cdots \ol{s_{i_n}}, \quad \dol{w} = \dol{s_{i_1}} \cdots \dol{s_{i_n}},
\]
where $s_{i_1} \cdots s_{i_n}$ is any reduced word for $w$.  

\begin{prop} \label{prop:gaussian}
\emph{(\cite[6.5.8, 7.4.11]{Kumar2002})} The multiplication map $N_- \times H \times N_+ \to G$ is a biregular isomorphism onto an open subvariety $G_0$.  Thus for any $g \in G_0$ we may write
\[
g = [g]_- [g]_0 [g]_+
\]
for some unique $[g]_{\pm} \in N_{\pm}$ and $[g]_0 \in H$.  Moreover, the maps 
\[
G_0\to N_{\pm} \: (\text{resp. } H), \quad g \mapsto [g]_{\pm} \: (\text{resp. } [g]_0)
\]
are regular.  
\end{prop}
\begin{prop}\label{prop:G_0eq}
\emph{(\cite[7.2]{Geiss2010})}
We have
\[
G_0 = \{ x \in G | \De^{\om_j}(x) \neq 0 \text{ for all } 1 \leq j \leq \wt{r} \},
\]
where the $\De^{\om_j}$ are the principal minors of \cref{prop:repint}.
\end{prop}

\begin{prop} \label{prop:Bruhat}
\emph{(\cite[7.4.2]{Kumar2002})} The group $G$ has positive and negative Bruhat decompositions
\[
G = \bigsqcup_{w \in W} B_+ \dot{w} B_+ = \bigsqcup_{w \in W} B_- \dot{w} B_-,
\]
where $\dot{w}$ is any representative of $w$ in $G$.  
\end{prop}

In particular, $G$ is a disjoint union of the \emph{double Bruhat cells}
\[
G^{u,v} := B_+ \dot{u} B_+ \cap B_- \dot{v} B_-.
\]

To obtain a more explicit description of the double Bruhat cells, we introduce the $\ell(w)$-dimensional unipotent subgroups
\begin{gather*}
N_{+}(w) := N_{+} \cap \dot{w} N_{-} \dot{w}^{-1}, \quad N_{-}(w) := N_{-} \cap \dot{w}^{-1} N_{+} \dot{w}
\end{gather*}
associated to any $w \in W$.  These have complementary infinite-dimensional subgroups
\[
N'_+(w) := N_+ \cap \dot{w} N_{+} \dot{w}^{-1}, \quad N'_-(w) := N_- \cap \dot{w}^{-1} N_{-} \dot{w}.
\]

\begin{prop}\label{prop:unipotentfactorization}
\emph{(\cite[6.1.3]{Kumar2002})} For any $w \in W$, the multiplication maps
\[
N_\pm(w) \times N'_\pm(w) \to N_\pm
\]
are biregular isomorphisms.
\end{prop}

The Bruhat decomposition then admits the following refinement:

\begin{corollary} \label{prop:refinedBruhat}
The natural maps
\[
N_+(w) \to N_+(w) \dot{w} B_+ / B_+, \quad N_-(w) \to B_- \backslash B_- \dot{w} N_-(w)
\]
are biregular isomorphisms.  In particular, the Bruhat cells can be written as
\[
B_+ \dot{w} B_+ = N_+(w) \dot{w} B_+, \quad B_- \dot{w} B_- = B_- \dot{w} N_-(w).
\]
\end{corollary}

\begin{corollary} \label{cor:FZ2.9/10}
For any $x \in B_+ \dot{w} B_+$, we have $\dot{w}^{-1} x \in G_0$.  Then
\[
\pi_+(x) := \dot{w}[\dot{w}^{-1} x]_- \dot{w}^{-1} \in N_+(w)
\]
and $x = \pi_+(x) \dot{w} b_+$ for some $b_+ \in B_+$.  Similarly, if $x \in B_- \dot{w} B_-$, then $x \dot{w}^{-1} \in G_0$,
\[
\pi_-(x) := \dot{w}^{-1} [x \dot{w}^{-1}]_+ \dot{w} \in N_-(w),
\]
and $x = b_- \dot{w} \pi_-(x)$ for some $b_- \in B_-$.  
\end{corollary}

\begin{prop}\label{prop:fdimdbc}
The map
\[
G^{u,v} \to N_+(u) \times N_-(v) \times H, \quad x \mapsto (\pi_+(x),\pi_-(x),[\ol{u}^{-1}x]_0)
\]
provides an isomorphism of $G^{u,v}$ with the open set 
\[
\{(n_+,n_-,h)|\ol{v}n_- n_+^{-1} \ol{u}^{-1} \in G_0 \} \subset N_+(u) \times N_-(v) \times H.
\]  In particular, $G^{u,v}$ is a rational affine variety of dimension $\ell(u)+\ell(v)+\wt{r}$.
\end{prop}

\begin{proof}
By an elementary calculation one checks that
\[
(n_+,n_-,h) \mapsto n_+ \ol{u} h [\ol{v} n_- n_+^{-1} \ol{u}^{-1}]_+
\]
provides the inverse map.  By \cref{prop:G_0eq} the given open set is the nonvanishing locus of the pullback of $\prod_{1 \leq j \leq \wt{r}}\De^{\om_j} \in \BC[G]$ along the regular map
\[
(n_+,n_-,h) \mapsto \ol{v}n_- n_+^{-1} \ol{u}^{-1}.
\]
The last statement then follows since $N_+(u) \times N_-(v) \times H$ is an open subvariety of $\BA^{\ell(u)+\ell(v)+\wt{r}}$.
\end{proof}

\subsection{Strongly Regular Functions and Generalized Minors} 
When $G$ is infinite-dimen\-sional, there are several natural algebras of functions one may consider on it.  Being an ind-variety, $G$ is the increasing union of finite-dimensional varieties, and the inverse limit of their coordinate rings is a complete topological algebra of functions on $G$.  For our purposes it is more practical to consider a proper subalgebra of this, the ring of strongly regular functions.  

Given a dominant integral weight $\la \in P_+$ we have an irreducible highest-weight $\fg$-module $L(\la)$ and its graded dual $L(\la)^\vee$, both of which integrate to representations of $G$.  Recall from \cref{sec:KMalg} that $L(\la)$ is equipped with a nondegenerate bilinear form.  For each $v_1, v_2 \in L(\la)$, we use this to define a function on $G$ by taking
\[
g \mapsto \langle v_1 | g \cdot v_2 \rangle.
\]
We regard this as a matrix coefficient of the image of $g$ in $\mathrm{End}\; L(\la)$.

\begin{defn}
(\cite{Kac1983}) The algebra of \emph{strongly regular} functions, which we will denote simply by $\BC[G]$, is the algebra generated by all such matrix coefficients of irreducible highest-weight representations.  
\end{defn}

\begin{prop}\label{prop:sreg}
\emph{(\cite[Theorem 1]{Kac1983})}
The algebra $\BC[G]$ is closed under the $G \times G$ action
\[
((g_1,g_2)\cdot f)(g) = f(g_1^{-1}g g_2).
\]
Furthermore, as $G \times G$-modules there is an isomorphism
\[
\BC[G] \cong \bigoplus_{\la \in P_+}(L(\la)^\vee \ot L(\la)).
\]
\end{prop}

\begin{defn}\label{prop:repint}
Given a fundamental weight $\om_i$ and a pair $w,w' \in W$, the \emph{generalized minor} $\De_{w,w'}^{\om_i}$ is the matrix coefficient
\[
g \mapsto \langle \ol{w} v_{\om_i} | g \ol{w'} v_{\om_i} \rangle,
\]
where $v_{\om_i}$ is a highest-weight vector of $L(\om_i)$.  The \emph{principal minor} $\De^{\om_i} := \De^{\om_i}_{e,e}$ is characterized by the fact that on the dense open set $G_0$,
\[
\De^{\om_i}: g=[g]_-[g]_0[g]_+ \mapsto [g]_0^{\om_i}.
\]
The other minors can then be expressed in terms of $\De^{\om_i}$ by
\[
\De^{\om_i}_{w,w'}(g) = \De^{\om_i}(\ol{w}^{-1}g\ol{w'}).
\]
\end{defn}

\begin{prop} \label{prop:prime}
The algebra $\BC[G]$ is a unique factorization domain in which the generalized minors are prime.  Two minors $\De_{u,v}^{\om_j}$ and $\De_{u',v'}^{\om_i}$ are relatively prime unless $u \om_j = u' \om_i$ and $v \om_j = v' \om_i$.
\end{prop}
\begin{proof}
That $\BC[G]$ is a unique factorization domain is Theorem 3 in \cite{Kac1983}, and the fact that the principal minors are prime is contained in the proof thereof.  Since an arbitrary generalized minor only differs from a principal minor by an automorphism of $\BC[G]$, it is also prime.

If $u \om_j = u' \om_i$ and $v \om_j = v' \om_i$, it is clear from \cref{prop:repint} that the generalized minors $\De_{u,v}^{\om_j}$ and $\De_{u',v'}^{\om_i}$ differ by a scalar multiple. On the other hand, if $u \om_j \neq u' \om_i$ or $v \om_j \neq v' \om_i$, it is clear from the decomposition in \cref{prop:sreg} that $\De_{u,v}^{\om_j}$ and $\De_{u',v'}^{\om_i}$ are linearly independent.  But the only units of $\BC[G]$ are the constant functions \cite[2.1c]{Kac1983}, so the proposition follows.
\end{proof}

The identity established in the next proposition plays a key role in the cluster algebras constructed on double Bruhat cells, providing the prototypical example of an exchange relation.  It is a direct generalization of \cite[1.17]{Fomin1999}, which in turn generalizes several classical determinantal identities.  The proof below follows that in \cite[1.17]{Fomin1999}, though when the Cartan matrix does not have full rank and $r < \wt{r} = \dim(H)$ it is important to use \cref{eq:extcartan} in interpreting the right-hand side of the identity.

\begin{prop} \label{prop:gendetid}
Suppose $u,v \in W$ satisfy $\ell(us_i) > \ell(u)$ and $\ell(vs_i) > \ell(v)$ for some $1 \leq i \leq r$.  Then
\[
\De^{\om_i}_{u,v}\De^{\om_i}_{us_i,vs_i} = \De^{\om_i}_{us_i,v}\De^{\om_i}_{u,vs_i} + \prod_{\substack{1 \leq k \leq \wt{r} \\ k \neq i}}(\De^{\om_k}_{u,v})^{-C_{ki}}.
\]
\end{prop}
\begin{proof}
It suffices to consider $u=v=e$.  In the case of arbitrary $u,v$, showing both sides are equal when evaluated at some $x \in G$ is then equivalent to showing both sides take the same value at $\ol{u}^{-1}x\ol{v}$ in the identity case.

Let
\[
f_1 = \De^{\om_i}_{e,e}\De^{\om_i}_{s_i,s_i} - \De^{\om_i}_{s_i,e}\De^{\om_i}_{e,s_i}, \quad f_2 = \prod_{\substack{1 \leq k \leq \wt{r} \\ k \neq i}}(\De^{\om_k}_{e,e})^{-C_{ki}}.
\]
We claim that $f_1$ and $f_2$ satisfy the following conditions, where we consider $\BC[G]$ as a $G \times G$ representation as in \cref{prop:sreg}:
\begin{enumerate}
\item
They are invariant under $N_- \times N_+$.
\item
They have weight $(\al_i - 2\om_i, 2\om_i - \al_i)$.
\item
They both evaluate to 1 at the identity.
\end{enumerate}
These conditions uniquely determine a function on the dense subset $G_0$, hence on all of $G$, so together imply the proposition.

The fact that $f_2$ satisfies the given conditions is essentially immediate; for (2) we must recall the definition of $C_{ij}$ for $r \leq j \leq \wt{r}$ in \cref{eq:extcartan}.  Likewise conditions (2) and (3) hold straightforwardly for $f_1$.

We claim then that $f_1$ is invariant under right translations by $N_+$.  Clearly it is invariant under right translation by $x_j(t)$ for $j \neq i$ and $t \in \BC$, so we need only show that it is invariant under right translations by $x_i(t)$.  

It is immediate that $\De_{e,e}^{\om_i}(xx_i(t)) =  \De_{e,e}^{\om_i}(x)$ and $\De_{s,e}^{\om_i}(xx_i(t)) =  \De_{s,e}^{\om_i}(x)$.  We claim further that
\begin{gather} \label{eq:desi}
\De_{e,s_i}^{\om_i}(xx_i(t)) = \De_{e,s_i}^{\om_i}(x) + t\De_{e,e}^{\om_i}(x), \\ \label{eq:desi2}
\De_{s_i,s_i}^{\om_i}(xx_i(t)) = \De_{s_i,s_i}^{\om_i}(x) + t\De_{s_i,e}^{\om_i}(x).
\end{gather}
To see this, first note that for a highest-weight vector $v_{\om_i}$ of $L(\om_i)$ we have
\begin{equation}\label{eq:SL2}
x_i(t)\ol{s_i} \cdot v_{\om_i} = \ol{s_i}\cdot v_{\om_i} + tv_{\om_i}.
\end{equation}
This is a simple computation in $SL_2$ representation theory; when we decompose $L(\om_i)$ as a $\varphi_i(SL_2)$-representation, $v_{\om_i}$ generates a copy of the standard $SL_2$-representation.  But now \cref{eq:desi,eq:desi2} follow immediately in light of \cref{prop:repint}, and we conclude that
\begin{align*}
f_1(xx_i(t)) & = \De_{e,e}^{\om_i}(x)( \De_{s_i,s_i}^{\om_i}(x) + t\De_{s_i,e}^{\om_i}(x) ) - \De_{s_i,e}^{\om_i}(x) (\De_{e,s_i}^{\om_i}(x) + t\De_{e,e}^{\om_i}(x)) \\
& = f_1(x).
\end{align*}

One easily checks that $f_1(x) = f_1(\si(x^{-1}))$, where $\si$ is the automorphism of $G$ induced from \cref{eq:si}. From this the right $N_+$-invariance of $f_1$ implies its left $N_-$-invariance, hence condition (1) indeed holds for $f_1$.
\end{proof}  

\section{Coordinates on Double Bruhat Cells} \label{sec:dbc}

When $G$ is a semisimple algebraic group, each double Bruhat cell $G^{u,v}$ is endowed with several natural families of coordinate systems.  To any double reduced word for $(u,v)$ is associated a parametrization of $G^{u,v}$ by one-parameter simple root subgroups, the definition of which is motivated by the theory of total positivity \cite{Fomin1999}.  In \cite{Fock2006}, a modified version of this parametrization was introduced on the adjoint form of $G$ using coweight subgroups; the resulting coordinates are convenient for working with the standard Poisson bracket, and transform as cluster $\CX$-coordinates as the double reduced word is varied. 

Explicitly describing the inverse maps to these parametrizations amounts to solving certain factorization problems in the group.  In the case of one-parameter simple root subgroups the solution was found in terms of twisted generalized minors in \cite{Fomin1999}.  In \cref{sec:dbcfact} we extend this result to the setting of symmetrizable Kac-Moody groups, after generalizing the various coordinates as necessary in \cref{sec:words}.  In \cref{sec:XandA} we use this to solve the corresponding factorization problem for the coweight parametrization.  In the process we will directly recover the entries of the exchange matrix defined in \cite{Berenstein2005}.

\subsection{Double Reduced Words and Parametrizations}\label{sec:words}

Let $G$ be a symmetrizable Kac-Moody group and $G^{u,v}$ a fixed double Bruhat cell.  A \emph{double reduced word} $\mb{i}=(i_1, \dots, i_m)$ for $(u,v)$ is a shuffle of a reduced word for $u$ written in the alphabet $\{-1,\dots,-r\}$ and a reduced word for $v$ written in the alphabet $\{1,\dots,r\}$.  

\begin{defn}\label{defn:Tmap}
Let $\mb{i}$ be a double reduced word for $(u,v)$, and set $m = \ell(u) + \ell(v)$.  Let $T_{\mb{i}}$ denote the complex torus $(\BC^*)^{m+\wt{r}}$ with coordinates $t_1,\dots,t_{m+\wt{r}}$.  Then we have a map $x_{\mb{i}}:T_{\mb{i}} \to G$ given by
\[
x_{\mb{i}}: (t_i, \dots , t_{m+\wt{r}}) \mapsto x_{i_1}(t_1) \cdots x_{i_m}(t_m)t_{m+1}^{\al^\vee_1}\cdots t_{m+\wt{r}}^{\al^\vee_{\wt{r}}}.
\]
Here $x_i(t)$ and $x_{-i}(t)$ denote the one-parameter subgroups corresponding to $\al_i$ and $-\al_i$, respectively.  When $G$ is an algebraic group this was defined in \cite{Fomin1999}, where the following result was also proved.
\end{defn}

\begin{prop}\label{prop:factorization}
The map $x_{\mb{i}}$ is an open immersion from $T_{\mb{i}}$ to $G^{u,v}$.
\end{prop}
\begin{proof}
First we show that the image of $x_{\mb{i}}$ is contained in $G^{u,v}$.  For each $1 \leq i \leq r$, we have $x_i(t) \in \CB_+$ and $x_{-i}(t) \in \CB_+ s_i \CB_+$.  Thus if ${k_1 < \dots < k_{\ell(u)}} \subset \{ 1,\dots,m \}$ are the indices of the negative entries in $\mb{i}$, 
\[
x_{\mb{i}}(t_1, \dots , t_{m+\wt{r}}) \in \CB_+ \cdots \CB_+ s_{i_{k_1}} \CB_+ \cdots \CB_+ s_{i_{k_{\ell(u)}}} \CB_+ \cdots \CB_+.
\]
Recall that for $w, w' \in W$, 
\[
\CB_+ w \CB_+ \cdot \CB_+ w' \CB_+ = \CB_+ w w' \CB_+
\]
whenever $\ell(ww') = \ell(w) + \ell(w')$ \cite[5.1.3]{Kumar2002}.  Thus in particular $x_{\mb{i}}(t_1, \dots , t_{m+\wt{r}}) \in \CB_+ u \CB_+$, and by the same argument $x_{\mb{i}}(t_1, \dots , t_{m+\wt{r}}) \in \CB_- v \CB_-$.

Suppose that
\[
x_{\mb{i}}(t_1, \dots , t_{m+\wt{r}}) = x_{\mb{i}}(t'_1, \dots , t'_{m+\wt{r}})
\] but $(t_1, \dots , t_{m+\wt{r}}) \neq (t'_1, \dots , t'_{m+\wt{r}})$, and let $k$ be the smallest index such that $t_k \neq t'_k$.  If $k > m$ this is a contradiction, since an element of $H$ factors uniquely as a product of coroot subgroups.  

On the other hand, if $k \leq m$, then $\mb{i}' := (i_k, \dots ,i_m)$ is a double reduced word for some $(u',v')$, and $x_{\mb{i}'}(t_k, \dots , t_{m+\wt{r}}) = x_{\mb{i}'}(t'_k, \dots , t'_{m+\wt{r}})$.  Multiplying both sides on the left by $x_{i_k}(-t'_k)$, we obtain
\[
x_{\mb{i}'}(t_k-t'_k, \dots , t_{m+\wt{r}}) = x_{\mb{i}''}(t'_{k+1}, \dots , t'_{m+\wt{r}}),
\]
where $\mb{i}'':=(i_{k+1}, \dots, i_m)$.  But by the first part of the proposition the left and right sides lie in different double Bruhat cells, hence by contradiction $x_{\mb{i}}$ must be injective.  But an injective regular map between smooth complex varieties of the same dimension is an open immersion, and the proposition follows.
\end{proof}

A closely related family of parametrizations was introduced in \cite{Fock2006} for semisimple algebraic groups.  Whereas so far we have taken $G$ to be simply-connected, to describe these $\CX$-coordinates we must consider its adjoint version.  When the Cartan matrix is not of full rank and the center of $G$ is positive-dimensional, we will abuse terminology and use $G_{\Ad}$ to denote a variant of the adjoint group.

Recall from \cref{sec:KMalg} that the fundamental weight basis of $P$ induces a dual basis of the cocharacter lattice $\Hom(\BC^*,H)$. We denote it by $\al_1^\vee,\dots,\al_{\wt{r}}^\vee$ since the first $r$ are exactly the coroots of $G$.  In parallel with this we define elements $\al_{r+1},\dots,\al_{\wt{r}}$ of $P$ by 
\[
\al_i = D \sum_{j=1}^r d_j^{-1}C_{ij}\om_j,
\]
where $D$ is the least common integer multiple of $d_1,\dots,d_r$.  Then $\oplus_{1\leq i \leq \wt{r}}\BZ \al_i$ is a full rank sublattice of $P$, and its kernel $\{h \in H | h^{\al_i} =1, 1 \leq i \leq \wt{r} \}$ is a discrete subgroup of the center of $G$.  We let $G_{\Ad}$ denote the quotient of $G$ by this discrete subgroup.  Of course, if $C$ has full rank this is exactly the adjoint form of $G$.

If $H_{\Ad}$ is the image of $H$ in $G_{\Ad}$, the character lattice of $H_{\Ad}$ is canonically isomorphic with $\oplus_{1\leq i \leq \wt{r}}\BZ \al_i$.  In particular, the cocharacter lattice of $H_{\Ad}$ inherits a dual basis $\om_1^\vee,\dots,\om_{\wt{r}}^\vee$ of \emph{fundamental coweights} such that $\langle \al_i | \om_j^\vee \rangle = \de_{i,j}$ for $1 \leq i,j \leq \wt{r}$.  We will denote elements of the corresponding one-parameter subgroups of $H_{\Ad}$ by $t^{\om_i^\vee}$, where $t \in \BC^*$; in other words, $t^{\om_i^\vee}$ is defined so that
\[
(t^{\om_i^\vee})^{\al_j} = t^{\de_{ij}}.
\]
We can now define $C_{ij} := \langle \al_j | \al^\vee_i \rangle$ for all $1 \leq i,j\leq \wt{r}$.  The definitions of $\al_i$ for $i > r$ are chosen exactly to obtain the following proposition, which the reader may easily verify.

\begin{prop}
The $\wt{r} \times \wt{r}$ integer matrix with entries $C_{ij}$ is nondegenerate and symmetrizable (with $d_i = D$ for $i > r$).  Moreover, the coweights and coroots are related by
\[
\al_i^\vee = \sum_{j = 1}^{\wt{r}} C_{ij} \om_j^\vee.
\]
\end{prop} 

\begin{example}
Let $G$ be the untwisted affine Kac-Moody group corresponding to a simply-connected simple algebraic group $\mathring{G}$.  That is, $G$ is the semidirect product of $\BC^*$ and the universal central extension of the group of regular maps from $\BC^*$ to $\mathring{G}$.  Then the center $Z(\mathring{G})$ of $\mathring{G}$ sits inside $G$ as constant maps, and we may choose the fundamental coweights so that $G_{\Ad} = G/Z(\mathring{G})$.
\end{example}

\begin{defn}\label{defn:Xdef}
Let $\mb{i} = (i_1,\dots,i_m)$ be a double reduced word for $(u,v)$, and let $I$ denote the index set $I=\{-\wt{r},\dots,-1\} \cup \{1,\dots,m\}$.  Let $\CX_{\mb{i}}$ denote the torus $(\BC^*)^I$ with coordinates $\{X_i\}_{i \in I}$.  We will write $E_{i} := x_i(1)$ for $i \in \{\pm 1,\dots,\pm r \}$. Then we have a map $x_{\mb{i}}: \CX_{\mb{i}} \to G_{\Ad}^{u,v}$ given by
\begin{gather*}\label{eq:Xmap}
x_{\mb{i}}: (X_{-\wt{r}},\dots,X_m) \mapsto X^{\om_{\wt{r}}^\vee}_{-\wt{r}}\cdots X^{\om_1^\vee}_{-1} E_{i_1} X^{\om_{|i_1|}^\vee}_1 \cdots E_{i_j} X^{\om_{|i_j|}^\vee}_j\cdots E_{i_m}X^{\om_{|i_m|}^\vee}_m.
\end{gather*}
\end{defn}
Though we have also used $x_{\mb{i}}$ to denote the map of \cref{defn:Tmap}, it will always be clear from the context which we mean.  The following proposition may be deduced straightforwardly from \cref{prop:factorization}.  
\begin{prop}\label{prop:fincov}
The map $x_{\mb{i}}: \CX_{\mb{i}} \to G_{\Ad}^{u,v}$ is an open immersion.  Moreover, the restriction of the quotient map $\pi_G: G^{u,v} \to G^{u,v}_{\Ad}$ to $T_{\mb{i}}$ is a finite covering of $\CX_{\mb{i}}$.
\end{prop}

In particular, the $t_i$ and  $X_i$ may be regarded as implicitly defined rational coordinates on $G^{u,v}$ and $G^{u,v}_{\Ad}$.  In \cite{Fomin1999}, the former coordinates were explicitly described in the semisimple case in terms of a certain family of generalized minors whose definition we now recall.

Given an index $1 \leq k \leq m$ and a double reduced word $\mb{i}$, we define two Weyl group elements
\[
u_{<k} := s^{\frac12(1-\ep_1)}_{i_1} \cdots s^{\frac12(1-\ep_{(k-1)})}_{i_{(k-1)}}, \quad v_{>k} :=  s^{\frac12(\ep_n+1)}_{i_n} \cdots s^{\frac12(\ep_{(k+1)}+1)}_{i_{k+1}},
\]
where $\ep_k$ is equal to $1$ if $i_k >0$ and $-1$ if $i_k < 0$.  In short, $u_{<k}$ is the part of the reduced word for $u$ whose indices in $\mb{i}$ are less than $k$, and $v_{>k}$ is the inverse of the part of the reduced word for $v$ whose indices in $\mb{i}$ are greater than $k$.  For purposes of the following definition, we will also set $v_{>k} = v^{-1}$ if $k < 0$.

\begin{defn}\label{defn:Aminors}
If $\mb{i} = (i_1,\dots,i_m)$ is a double reduced word for $(u,v)$, let $I$ denote the index set $\{-\wt{r},\dots,-1\} \cup \{1,\dots,m\}$ and let $i_k = k$ for $k<0$.  Then to each $k \in I$ we associate a generalized minor
\[
A_{k,\mb{i}} := \De^{\om_{|i_k|}}_{u_{\leq k},v_{>k}}.
\]
When the choice of double reduced word is clear we will abbreviate this to $A_k$.  
\end{defn}

\begin{remark}
One may define the postive part $G^{u,v}_{>0}$ of $G^{u,v}$ as the image of $\BR^{m+\wt{r}}_{>0} \subset T_{\mb{i}}$ in $G^{u,v}$; when $G$ is a semisimple algebraic group this is an important object in the theory of total positivity, the study of which motivated the work \cite{Fomin1999}.  Though total positivity will not play a direct role in the present article, we note in passing that the above definition of $G^{u,v}_{>0}$ agrees with the analogous definition in terms of the coweight parametrization.  That is, if $g \in G^{u,v}_{>0}$ it follows straightforwardly that $\pi_G(g) \in G^{u,v}_{\Ad}$ is in the image of $\BR^{m+\wt{r}}_{>0} \subset \CX_{\mb{i}}$.
\end{remark}

\subsection{The Twist Isomorphism}
To precisely describe the relationships among the various coordinates introduced in \cref{sec:words}, we will require a certain isomorphism of inverse double Bruhat cells, called the twist map in \cite{Fomin1999}.  In this section we recall its key properties, which extend readily to the setting of Kac-Moody groups.  

\begin{defn} We write $x \mapsto x^\th$ for the automorphism of $G$ which acts as follows on the Cartan subgroup and Chevalley generators:
\[
a^\th=a^{-1} \quad (a \in H), \quad x_i(t)^\th = x_{-i}(t) \quad (1 \leq i \leq r).
\]
\end{defn}

\begin{defn}
For any $u, v \in W$, the \emph{twist map} $\ze^{u,v}: G^{u,v} \to G^{u^{-1},v^{-1}}$ is defined by
\begin{equation} \label{eq:twistdef}
\ze^{u,v}: x \mapsto \left([\ol{u}^{-1}x]_-^{-1}\ol{u}^{-1}x\ol{v^{-1}}[x\ol{v^{-1}}]_+^{-1}\right)^\th.
\end{equation}
\end{defn}

\begin{prop} \label{prop:twist}
The twist map $\ze^{u,v}$ is an isomorphism of $G^{u,v}$ and $G^{u^{-1},v^{-1}}$ whose inverse is $\ze^{u^{-1},v^{-1}}$.
\end{prop}
\begin{proof}
That $\ze^{u,v}$ is well-defined on $G^{u,v}$ follows from \cref{cor:FZ2.9/10}.  To see that $x' = \ze^{u,v}(x) \in B_- \dot{v}^{-1} B_-$, we simplify \cref{eq:twistdef} as
\[
x' = \left( [\ol{u}^{-1} x]_0 [\ol{u}^{-1} x]_+ y_-^{-1} \right)^\th \ol{v}^{-1} \in G_0 \ol{v}^{-1},
\]
where $y_- = \pi_-(x)$ as in \cref{cor:FZ2.9/10}.
In particular, 
\begin{equation} \label{eq:x'v}
[x' \ol{v}]_+ = (y_-^{-1})^\th \in N_-(v)^\th = N_+(v^{-1}),
\end{equation}
hence $x' \in B_- \dot{v}^{-1} B_-$.  Similarly one can see that
\begin{equation} \label{eq:ux'}
[\dol{u} x']_- = (y_+^{-1})^\th \in N_-(u^{-1}),
\end{equation}
hence $x' \in B_+ \dot{u}^{-1} B_+$.  But now the fact that $\ze^{u,v}$ and $\ze^{u^{-1},v^{-1}}$ are inverse to each other follows from plugging our expressions for $[x' \ol{v}]_+$ and $[\dol{u} x']_-$ into the definition of $\ze^{u^{-1},v^{-1}}$ and simplifying.
\end{proof}

\begin{prop} \label{prop:twist_0}
The twist map $\ze^{u,v}$ restricts to an isomorphism of the open sets $G_0^{u,v}$ and $G_0^{u^{-1},v^{-1}}$.  Moreover, if $x \in G_0^{u,v}$, $x' = \ze^{u,v}(x)$, we have
\begin{equation} \label{eq:x'_0}
[x']_0 = [\ol{u}^{-1} x]_0^{-1} [x]_0 [x \ol{v^{-1}}]_0^{-1}.
\end{equation}
\end{prop}
\begin{proof}
We can rewrite \cref{eq:twistdef} as 
\[
x' = \left( [\ol{u}^{-1} x]_0 [\ol{u}^{-1} x]_+ x^{-1} [x \ol{v^{-1}}]_- [x \ol{v^{-1}}]_0 \right)^\th,
\]
and the proposition follows from taking the Cartan part of each side.
\end{proof}

If  $w = s_{i_1}\cdots s_{i_{\ell(w)}}$ is a reduced word for $w \in W$, we define Weyl group elements
\[
w_{<k} := s_{i_1}\cdots s_{i_{k-1}}, \quad w_{>k} := s_{i_{\ell(w)}}\cdots s_{i_k},
\]
and similarly $w_{\leq k}$, $w_{\geq{k}}$.  

\begin{prop} \label{prop:newlem}
If $x \in G_0^{u,v}$, $x' = \ze^{u,v}(x)$, and $1 \leq j \leq \wt{r}$,
\[
\De_{v_{>k},e}^{\om_j}(y_-) = \frac{\De^{\om_j}_{e,v_{\leq k}}(x')}{\De^{\om_j}_{e,v}(x')}, \quad \De^{\om_j}_{e,u_{<k}}(y_+) = \frac{\De^{\om_j}_{u_{\geq k },e}(x')}{\De^{\om_j}_{u^{-1},e}(x')}.
\]
\end{prop}
\begin{proof}
First we claim that if $y_\pm = \pi_\pm(x)$ and $y'_\pm = \pi_\pm(x')$, then
\begin{align*}
y'_+ = \dol{u}^{-1} (y_+^{-1})^\th \dol{u}, \quad
y'_- = \ol{v} (y_-^{-1})^\th \ol{v}^{-1}.
\end{align*}
This follows straightforwardly from \cref{eq:x'v} and \cref{eq:ux'}.

We can use these identities to write
\[
\De_{v_{>k},e}^{\om_j}(y_-) = \De^{\om_j}(\dol{v_{\leq k}}^{-1}\dol{v}y_-) = \De^{\om_j}(\dol{v_{\leq k}}^{-1}(y_-^{'-1})^\th\dol{v}).
\]
One can check that $\De^{\om_j}((g^{-1})^\th) = \De^{\om_j}(g)$ for all $g \in G$, hence
\[
\De^{\om_j}(\dol{v_{\leq k}}^{-1}(y_-^{'-1})^\th\dol{v}) = \De^{\om_j}(\ol{v}^{-1}y'_-\ol{v_{\leq k}}).
\]
By \cref{cor:FZ2.9/10}, $x' = b_- \ol{v}^{-1} y'_-$ for some $b_- \in B_-$.  Then
\[
\De^{\om_j}(\ol{v}^{-1}y'_-\ol{v_{\leq k}}) = \De^{\om_j}(b_-^{-1} x' \ol{v_{\leq k}}) = [b_-]_0^{-\om_j}\De^{\om_j}(x' \ol{v_{\leq k}}).
\]
Now since $\ol{v}^{-1}y'_-\ol{v} \in N_+$,
\[
\De^{\om_j}_{e,v}(x') = \De^{\om_j}(b_- \ol{v}^{-1}y'_-\ol{v}) = [b_-]_0^{\om_j}.
\]
But then
\[
[b_-]_0^{-\om_j}\De^{\om_j}(x' \ol{v_{\leq k}}) = \frac{\De^{\om_j}_{e,v_{\leq k}}(x')}{\De^{\om_j}_{e,v}(x')},
\]
proving the first part of the proposition.  The remaining statement then follows by essentially the same argument.
\end{proof}

\subsection{Factorization in Unipotent Groups}\label{sec:groupfact} 
In \cref{thm:main} we derive expressions for the $t_i$ as Laurent monomials in the twists of the $A_i$, generalizing the main result of \cite{Fomin1999} to the Kac-Moody setting.  The strategy of the proof is the same as in the finite-dimensional case.  We build up to the main theorem by solving a series of more elementary factorization problems, starting with the factorization of the unipotent subgroup $N_-(w)$ as a product of one-parameter subgroups.  This in turn lets us solve the factorization problem for the unipotent cell $N_+^w := N_+ \cap B_- \dot{w} B_-$.  From here we can extract the solution for a general double Bruhat cell by reducing to the case of an ``unmixed'' double reduced word.  

For $w \in W$, recall the unipotent group $N_-(w) = N_- \cap \dot{w}^{-1} N_+ \dot{w}$ and fix a reduced word $w = s_{i_1} \cdots s_{i_n}$.  For short we will write
\[
w_k := w_{\geq k} = s_{i_n}\cdots s_{i_k}.
\]
Now define one-parameter subgroups
\[
y_k(p_k) = \ol{w_{k+1}}x_{-i_k}(p_k)\ol{w_{k+1}}^{-1},
\]
where we take $w_{n+1}=e$.

\begin{lemma}\label{lem:yk}
For any $p_k \in \BC$ we have
\[
\ol{w_m}^{-1}y_k(p_k)\ol{w_m} \in
\begin{cases}
N_- & \quad m > k \\
N_+ & \quad m \leq k. \\
\end{cases}
\]
\end{lemma}
\begin{proof}
Follows straightforwardly from the standard fact that if $\ell(ws_i) > \ell(w)$ for some $w \in W$, then $w(\al_i)$ is again a positive root.
\end{proof}

\begin{prop}\label{prop:groupfact}
The map $y_{\mb{i}}: \BC \to N_-(w)$ given by
\[
(p_1,\dots,p_n) \mapsto y = y_1(p_1)\cdots y_n(p_n)
\]
is an isomorphism.  Its inverse is given explicitly by
\[
p_k = \De^{\om_{i_k}}_{w_k, w_{k+1}}(y).
\]
\end{prop}
\begin{proof}
That $y_{\mb{i}}$ is an isomorphism is well-known \cite[5.2]{Geiss2010}.  Let $y_k = y_k(p_k)$ be as in \cref{lem:yk}, and
\[	y_{<k} = y_1\cdots y_{k-1}, \quad y_{>k}=y_{k+1}\cdots y_n.	\]
In particular,
\[
y = y_{<k}\cdot y_k \cdot y_{>k}.
\]
It follows from \cref{lem:yk} that 
\[
\ol{w_k}^{-1} y_{<k} \ol{w_k} \in N_-, \quad \ol{w_{k+1}}^{-1} y_{>k}\ol{w_{k+1}} \in  N_+.
\]
But we then have
\begin{align*}
\De_{w_k,w_{k+1}}^{\om_{i_k}}(y) & = \De^{\om_{i_k}}((\ol{w_k}^{-1}y_{<k}\ol{w_k})\ol{w_k}^{-1}y_k \ol{w_{k+1}}(\ol{w_{k+1}}^{-1}y_{>k} \ol{w_{k+1}})) \\
& = \De^{\om_{i_k}}(\ol{w_k}^{-1}y_k \ol{w_{k+1}}) \\
& = \De^{\om_{i_k}}(\ol{s_{i_k}}^{-1}x_{-i_k}(p_k)) \\
& = p_k.
\end{align*}
The first two lines follow from the definitions of the generalized minors, while the last is a simple computation in $SL_2$ representation theory (similar to \cref{eq:SL2}).
\end{proof}

\subsection{Factorization in Unipotent Cells}
We can now solve the factorization problem for the unipotent cell $N^w_+ := N_+ \cap B_- \dot{w} B_-$.  Given a reduced word $w = s_{i_1}\cdots s_{i_n}$, $N^w_+$ has a birational parametrization
\[
(\BC^*)^{n} \to N_+^w, \quad (t_1,\dots,t_{n}) \mapsto x_{i_1}(t_1)\cdots x_{i_{n}}(t_{n}).
\]
The inverse map is described in \cref{prop:unifact}, which relies on the following two lemmas.

\begin{lemma} \label{lem:simplefactor}
Let $1 \leq i \leq r$.  Then any $x \in N_-$ can be written as $\ol{s_i} x' \ol{s_i}^{-1} x_{-i}(t)$ for some $x' \in N_-$ and $t \in \BC$.  Morevover, $t$ is given by
\[
t = \De_{s_i,e}^{\om_i}(x).
\]
\end{lemma}
\begin{proof}
That $g$ admits such an expression is an immediate consequence of \cref{prop:unipotentfactorization}. To verify that $t$ is given by the stated formula, we check that
\begin{align*}
\De_{s_i,e}^{\om_i}(x) & = \De^{\om_i}(x' \ol{s_i} x_{-i}(t)) \\
& = \De^{\om_i}(\ol{s_i} x_{-i}(t)) \\
& = t.
\end{align*}
The last line is another simple $SL_2$ computation.
\end{proof}

\begin{lemma} \label{lem:t1}
Let $x = x_{i_1}(t_1) \cdots x_{i_n}(t_n) \in N_+^w$ and $x' = x_{i_2}(t_2) \cdots x_{i_n}(t_n) \in N_+^{w'}$. Here $w' = s_{i_1}w$, and $\mb{i}' = (i_2, \dots, i_n)$ is a reduced word for $w'$.  Let $p_2,\dots,p_n$ be complex numbers such that $y' = \pi_-(x') = y_{\mb{i}'}(p_2,\dots,p_n)$.
Then
\[
y= \pi_-(x) = y_{\mb{i}}(p_1,\dots,p_n),
\]
where
\[
p_1 := \De_{s_{i_1},e}^{\om_{i_1}}(x_{-i_1}([\dol{w'}y']_0^{-\al_{i_1}}t_1^{-1})[\dol{w'}y']_-^{-1}).
\]
Moreover, $t_1$ can be recovered as
\[
t_1 = [\dol{w'}y']_0^{\om_{i_1} - \al_{i_1}}[\dol{w}y]_0^{-\om_{i_1}}.
\]
\end{lemma}
\begin{proof}
We denote $y_{\mb{i}}(p_1,\dots,p_n)$ by $\wt{y}$ during the proof.  To show $y = \wt{y}$ it suffices to show that $\dol{w}\wt{y} \in G_0$ and $[\dol{w}\wt{y}]_+ = x$, or equivalently that $\dol{w}\wt{y}x^{-1} \in B_-$.  
Now one can calculate that
\begin{align}\label{eq:wyx}
\dol{w}\wt{y}x^{-1} & = \ol{s_{i_1}}^{-1}x_{-i_1}(p_1)[\dol{w'}y']_-[\dol{w'}y']_0 x_1(-t_1).
\end{align}
Applying \cref{lem:simplefactor} to $x_{-i_1}([\dol{w'}y']_0^{-\al_{i_1}}t_1^{-1})[\dol{w'}y']_-^{-1}$, we know that
\[
x_{-i_1}([\dol{w'}y']_0^{-\al_{i_1}}t_1^{-1})[\dol{w'}y']_-^{-1} = \ol{s_{i_1}} y'' \ol{s_{i_1}}^{-1} x_{-{i_1}}(p_1)
\]
for some $y'' \in N_-$.  Combining this with \cref{eq:wyx} lets us write
\begin{align*}
\dol{w}\wt{y}x^{-1} & = (y'')^{-1}\ol{s_{i_1}}^{-1} x_{-i_1}([\dol{w'}y']_0^{-\al_{i_1}}t_1^{-1}) [\dol{w'}y']_0 x_{i_1}(-t_1) \\
& = (y'')^{-1}\ol{s_{i_1}}^{-1} [\dol{w'}y']_0 x_{-i_1}(t_1^{-1}) x_{i_1}(-t_1) \\
& = (y'')^{-1}\ol{s_{i_1}}^{-1} [\dol{w'}y']_0 \ol{s_{i_1}} t_1^{-\al^\vee_{i_1}} x_{-{i_1}}(-t_1^{-1}) \in B_-. \\
\end{align*}
The last line can be checked directly in $\varphi_{i_1}(SL_2)$.

If we take the $H$-components of each side, we see further that
\[
[\dol{w}y]_0 = \ol{s_{i_1}}^{-1} [\dol{w'}y']_0 \ol{s_{i_1}} t_1^{-\al^\vee_{i_1}}.
\]
The last assertion then follows by applying the character $\om_{i_1}$ to each side.
\end{proof}

\begin{prop}\label{prop:unifact}
Let $t_1, \dots, t_n$ be nonzero complex numbers and let $x = x_{i_1}(t_1) \cdots x_{i_n}(t_n) \in N_+^w$. Then
\[
t_k = \frac{1}{\De^{\om_{i_k}}_{w_k,e}(y)\De^{\om_{i_k}}_{w_{k+1},e}(y)}\prod_{\substack{1 \leq j \leq \wt{r} \\ j \neq i_k}}(\De^{\om_{j}}_{w_{k+1},e}(y))^{-C_{j,i_k}},
\]
where $y = \pi_-(x) \in N_-(w)$ and $w_k = s_{i_n}\cdots s_{i_k}$.
\end{prop}
\begin{proof}
Let
\[
x_{\geq k} := x_{i_k}(t_k)\cdots x_{i_n}(t_n), \quad y_{\geq k} = \ol{w_k}[x_{\geq k} \ol{w_k}]_+ \ol{w_k}^{-1}, \quad z_{\geq k} = \ol{w_k}^{-1}y_{\geq k}.
\]
Then applying \cref{lem:t1} to $x_{\geq k}$ we obtain
\[
t_k = [z_{\geq (k+1)}]_0^{\om_{i_k}-\al_{i_k}}[z_{\geq k}]_0^{-\om_{i_k}}.
\]
We claim then that $[z_{\geq k}]_0 = [\ol{w_k}^{-1}y]_0$.  This follows from
\begin{align*}
\ol{w_k}^{-1}y & = (\ol{w_k}^{-1}y_{<k} \ol{w_k}) \ol{w_k}^{-1}y_{\geq k} \\
& = (\ol{w_k}^{-1}y_{<k} \ol{w_k}) z_{\geq k},
\end{align*}
and the observation that 
\[
(\ol{w_k}^{-1}y_{<k} \ol{w_k}) \in N_-
\]
which follows from \cref{lem:yk}.
But then
\begin{align*}
t_k & = [\ol{w_{k+1}}^{-1}y]_0^{(\om_{i_k}-\al_{i_k})}[\ol{w_k}^{-1}y]_0^{-\om_{i_k}} \\
& = [\ol{w_{k+1}}^{-1}y]_0^{(\om_{i_k}-\sum_{1\leq j \leq \wt{r}}C_{j,i_k}\om_j)}[\ol{w_k}^{-1}y]_0^{-\om_{i_k}} \\
& = [\ol{w_{k+1}}^{-1}y]_0^{(-\om_{i_k}-\sum_{j \neq i_k}C_{j,i_k}\om_j)}[\ol{w_k}^{-1}y]_0^{-\om_{i_k}} \\
& = \frac{1}{\De^{\om_{i_k}}_{w_k,e}(y)\De^{\om_{i_k}}_{w_{k+1},e}(y)}\prod_{\substack{1 \leq j \leq \wt{r} \\ j \neq i_k}}(\De^{\om_{j}}_{w_{k+1},e}(y))^{-C_{j,i_k}},
\end{align*}
completing the proof.
\end{proof}

\subsection{Factorization in Double Bruhat Cells}\label{sec:dbcfact}

We now turn to the factorization problem in an arbitrary double Bruhat cell $G^{u,v}$.  

Let $\mb{i}=(i_1,\dots,i_m)$ be a double reduced word for $(u,v)$.  For $1 \leq j \leq m$ and $k \in I = \{-\wt{r},\dots,-1\} \cup \{1,\dots,m\}$, we define\footnote{Recall that if $P(x_1,\dots)$ is a boolean function of some variables $\{x_1,\dots\}$, $[P(x_1,\dots)]$ denotes the integer-valued function of the $x_i$ whose value is 1 when $P$ is true and 0 when $P$ is false.}  
\begin{align*}
\Psi_{j,k} := -\ep_j \ep_k\bigl([j=k] + [j = k^+]\bigr) + \frac{C_{|i_k|,|i_j|}}{2}\biggl(\ep_j(\ep_{k^+} - \ep_k)[k^+<j] - (1+\ep_j \ep_k)[k<j<k^+] \biggr);
\end{align*}
let us explain the notation.  For an index $k \in I$, we let
\[
k^+ := \mathrm{min}\{ \ell \in I : \ell > k, |i_\ell| = |i_k| \},
\]
setting $k^+ = m+1$ if there are no such $\ell$ (recall that we set $i_k = k$ for $k<0$).  Also recall that $\ep_k$ is equal to $1$ if $i_k>0$ and $-1$ if $i_k < 0$, with $\ep_{m+1} = 1$ for purposes of the above formula.  Note that $\Psi_{j,k}$ can only take the values $0$, $\pm 1$, and $\pm C_{|i_k|,|i_j|}$.

For $k \in I$, recall the generalized minors
\[
A_{k} := A_{k,\mb{i}} = \De^{\om_{|i_k|}}_{u_{\leq k},v_{>k}}
\]
from \cref{defn:Aminors}.  We let $x \mapsto x^\io$ denote the involutive antiautomorphism of $G$ determined by
\[
a^\io = a^{-1} \text{ for } a \in H,\quad x_i(t)^\io = x_i(t) \text{ for } 1 \leq i \leq r.
\]
It is clear that $\io$ restricts to an isomorphism of $G^{u,v}$ and $G^{u^{-1},v^{-1}}$, hence in particular $\ze^{u^{-1},v^{-1}}\circ \io$ is an automorphism of $G^{u,v}$.

\begin{thm}\label{thm:main}
Let $G$ be a symmetrizable Kac-Moody group, $u,v \in W$, and $\mb{i} = (i_1,\dots,i_m)$ a double reduced word for $(u,v)$.  Then if $x = x_{\mb{i}}(t_1,\dots,t_{m+\wt{r}})$ and $x' = (\ze^{u^{-1},v^{-1}}\circ \io) (x)$, we have
\begin{align}\label{eq:t_k}
t_j & = \prod_{k \in I}A_k(x')^{\Psi_{j,k}} 
\end{align}
for $1 \leq j \leq m$, and 
\begin{align}\label{eq:aom}
t_{m+j} & = \prod_{\substack{k \in I \\ |i_k| = j}}A_k(x')^{\frac12(\ep_{k^+}-\ep_k)}. 
\end{align}
for $1 \leq j \leq r$.\footnote{Though equivalent to \cite[Theorem 1.9]{Fomin1999} in finite type, the formulation here differs slightly to better match the conventions of \cite{Berenstein2005}.  The statement in \cite{Fomin1999} does not involve $\io$, and correspondingly the $t_i$ are expressed in terms of cluster variables on the inverse double Bruhat cell $G^{u^{-1},v^{-1}}$.  Also, our definition of $\Psi_{j,k}$ differs from the corresponding definition in \cite{Fomin1999} in order to facilitate the proof of \cref{prop:XtoA}.}
\end{thm}

\begin{proof}
The double reduced word $\mb{i} = (i_1,\dots,i_m)$ for $(u,v)$ induces an opposite double reduced word $\mb{i}^{\op} = (j_1,\dots,j_m)$ for $(u^{-1},v^{-1})$, by setting $j_k = i_{m+1-k}$.  Let $k^{\op} := m+1-k$ and $t'_k := t_{k^{\op}}$, so that
\[
x^{\io} = t_{m+\wt{r}}^{-\al^\vee_{\wt{r}}}\cdots t_{m+1}^{-\al^\vee_1}x_{j_1}(t'_1)\cdots x_{j_m}(t'_m).
\]

We first consider the case where $\mb{i}$ is ``unmixed''; that is, $k<\ell$ whenever $\ep_k > 0$ and $\ep_\ell < 0$.   Then $x^{\io} \in G_0^{u,v}$ and $[x^\io]_0 = t_{m+\wt{r}}^{-\al^\vee_{\wt{r}}}\cdots t_{m+1}^{-\al^\vee_1}$.  By \cref{prop:twist,prop:twist_0} we have
\[
t_{m+j} = [x^\io]_0^{-\om_j} = [\ol{u}^{-1} x']_0^{\om_j} [x']_0^{-\om_j} [x'\ol{v^{-1}}]_0^{\om_j}.
\]
One can then check that this agrees with \cref{eq:aom} in this case.  

Next observe that since $\mb{i}$ is unmixed, $y_- := \pi_-(x^\io)$ is equal to $\pi_-([x^\io]_+)$, and
\[
[x^\io]_+ = x_{j_{\ell(v)^\op}}(t'_{\ell(v)^\op})\cdots x_{j_m}(t'_m) \in N_+^{v^{-1}}.
\]
For $1\leq k \leq \ell(v)$, we can use \cref{prop:unifact} to obtain
\[
t_k = t'_{k^\op} = \frac{1}{\De^{\om_{i_k}}_{(v^{-1})_{>(k+1)^\op},e}(y_-)\De^{\om_{i_k}}_{(v^{-1})_{>k^\op},e}(y_-)}\biggl( \prod_{\substack{1 \leq j \leq \wt{r} \\ j \neq i_k}}(\De^{\om_{j}}_{(v^{-1})_{>k^\op},e}(y_-))^{-C_{j,|i_k|}} \biggr).
\]
Applying \cref{prop:newlem} to each term and using the observation that $(v^{-1})_{\leq k^\op} = v_{\geq k}$, we can rewrite this as
\[
t_k = \frac{1}{\De^{\om_{i_k}}_{e,v_{\geq (k+1)}}(x')\De^{\om_{i_k}}_{e,v_{\geq k}}(x')}\biggl( \prod_{\substack{1 \leq j \leq \wt{r} \\ j \neq i_k}} \De^{\om_j}_{e,v_{\geq k}}(x')^{-C_{j,|i_k|}} \biggr) \biggl( \prod_{1\leq j \leq \wt{r}}\De^{\om_j}_{e,v^{-1}}(x')^{C_{j,|i_k|}} \biggr)
\]
Using the fact that $\mb{i}$ is unmixed, one checks that this is equivalent to 
\[
t_k = A_k(x')^{-1}A_{k^-}(x')^{-1} \biggl( \prod_{\substack{ \ell \in I \\ \ell < k < \ell^+}}A_\ell(x')^{-C_{|i_\ell|,|i_k|}} \biggr) \biggl( \prod_{1 \leq j \leq \wt{r}}A_{-j}(x')^{C_{j,|i_k|}} \biggr).
\]
Here $k^- \in I$ is defined by $(k^-)^+ = k$.  Again, the reader may check that this expression agrees with \cref{eq:t_k} in this case.

For $\ell(v) <k \leq m$, we note that $\pi_+(x^\io) = \pi_+([x^\io]_-)$ and if $a = t_{m+1}^{\al^\vee_1}\cdots t_{m+\wt{r}}^{\al^\vee_{\wt{r}}}$,
\[
[x^\io]_- = x_{j_1}(a^{\al_{|j_1|}}t'_1)\cdots x_{j_{\ell(u)}}(a^{\al_{|j_{\ell(u)}|}}t'_{\ell(u)}).
\]
From here \cref{eq:t_k} follows by a similar argument as above, again invoking \cref{prop:unifact,prop:newlem}.  One arrives at
\begin{align*}
t_k & = A_k(x')^{-1}A_{k^-}(x')^{-1} \biggl( \prod_{\substack{ \ell \in I \\ \ell < k < \ell^+}}A_\ell(x')^{-C_{|i_\ell|,|i_k|}} \biggr) \biggl( \prod_{\ell: \ell^+ >m}A_\ell(x')^{C_{|i_\ell|,|i_k|}} \biggr) \times \\ 
& \quad \quad \times \biggl( \prod_{\ell \in I} A_\ell(x')^{-\frac12 C_{|i_\ell|,|i_k|}(\ep_{\ell^+} - \ep_\ell} \biggr),
\end{align*}
which agrees with \cref{eq:t_k} given that $\mb{i}$ is unmixed.  

Now suppose two double reduced words $\mb{i}$ and $\mb{i}'$ differ only by the exchange of two consecutive positive and negative indices.  That is, for some $1 \leq k < m$ and $1 \leq i,j \leq r$ we have
\[
i_k = i_{k+1}' = j,\quad i_{k+1} = i_k' = -i.
\]
We claim that if the theorem holds for $\mb{i}$ it also holds for $\mb{i'}$. Specifically, suppose that
\[
x = x_{\mb{i}}(t_1,\dots,t_{m+\wt{r}}) = x_{\mb{i'}}(t'_1,\dots,t'_{m+\wt{r}}),
\]
and that the $t_\ell$ satisfy \cref{eq:t_k,eq:aom}.  Then we claim the $t'_\ell$ also satisfy \cref{eq:t_k,eq:aom} with respect to the $A_{\ell,\mb{i}'}$.  

This is trivial unless $i = j$.  In that case, a straightforward computation in $\varphi_i(SL_2)$ yields that
\begin{gather*}
t'_{m+i} = t_{m+i}(1+t_k t_{k+1}), \quad t'_{m+\ell} = t_{m+\ell} \text{ for } \ell \neq i, \\
t'_\ell = t_\ell \text{ for } \ell < k, \quad t'_\ell = t_\ell(1+t_k t_{k+1})^{\ep_\ell C_{|i,i_\ell|}}, \text{ for } k+1 < \ell \leq m, \\
t'_k = t_{k+1}(1+t_k t_{k+1})^{-1}, \quad t'_{k+1} = t_k(1+t_k t_{k+1}).
\end{gather*}
Using the expression for $(1+t_k t_{k+1})$ provided by \cref{lem:switch} and simplifying the result, one can then check directly that \cref{eq:t_k,eq:aom} hold for the $t'_\ell$. But then since the image of $x_{\mb{i}}$ intersects the image of $x_{\mb{i'}}$ along a dense subset, we conclude that \cref{eq:t_k,eq:aom} hold for all points in the image of $x_{\mb{i'}}$.
\end{proof}

\begin{lemma} \label{lem:switch}
Suppose \cref{thm:main} holds for a double reduced word $\mb{i}$ with $i_k = -i_{k+1} = i$ for some $1 \leq i \leq r$.  Let $\mb{i}'$ be the double reduced word obtained by exchanging $i_k$ and $i_{k+1}$.  Then for $x = x_{\mb{i}}(t_1,\dots,t_{m+\wt{r}})$ and $x' = \ze^{u^{-1},v^{-1}} \circ \io$ we have
\[
1 + t_k t_{k+1} = \frac{A_{k,\mb{i}}(x') A_{k,\mb{i}'}(x')}{A_{k^-,\mb{i}}(x') A_{k+1,\mb{i}}(x')}.
\]
\end{lemma}
\begin{proof}
Letting $u' = u_{<k}, v' = v_{>(k+1)}$, we first calculate that
\begin{gather*}
A_{k,\mb{i}} = \De_{u',v'}^{\om_i}, \quad A_{k,\mb{i}'} = \De_{u's_i,v's_i}^{\om_i}, \\
A_{k+1,\mb{i}} = \De_{u's_i,v'}^{\om_i}, \quad A_{k^-,\mb{i}} = \De_{u',v's_i}^{\om_i}.
\end{gather*} 

Using \cref{eq:t_k} and the fact that $\ep_k = -\ep_{k+1} = 1$, we also have
\begin{align*}
1 + t_k t_{k+1} & = 1 + A_{k+1,\mb{i}}(x')^{-1}A_{k^-,\mb{i}}(x')^{-1}\biggl( \prod_{\ell < k < \ell^+}A_{\ell,\mb{i}}(x')^{-C_{|i_\ell|,i}} \biggr) \\
& =\frac{\De_{u's_i,v'}^{\om_i}(x')\De_{u',v's_i}^{\om_i}(x') + \prod_{\substack{1 \leq j \leq \wt{r} \\ j \neq i}}\De_{u',v'}^{\om_j}(x')^{-C_{ji}}}{\De_{u's_i,v'}^{\om_i}(x')\De_{u',v's_i}^{\om_i}(x')}.
\end{align*}
But then by \cref{prop:gendetid} this yields
\[
1 + t_kt_{k+1} = \frac{\De_{u',v'}^{\om_i}(x')\De_{u's_i,v's_i}^{\om_i}(x')}{\De_{u's_i,v'}^{\om_i}(x')\De_{u',v's_i}^{\om_i}(x')},
\]
and the lemma follows.
\end{proof}

\subsection{$\CX$-coordinates and Generalized Minors}\label{sec:XandA}
Recall that the coweight parametrization $x_{\mb{i}}: \CX_{\mb{i}} \to G_{\Ad}^{u,v}$ of \cref{defn:Xdef} yields a set $\{X_i\}_{i \in I}$ of rational coordinates on $G_{\Ad}^{u,v}$. Since the image of $T_{\mb{i}}$ in $G^{u,v}$ is a finite cover of $\CX_i$ in $G^{u,v}$, the pullbacks of the $X_i$ to $G^{u,v}$ are Laurent monomials in the $t_i$, and, by \cref{thm:main}, in the twisted generalized minors.  In this section we derive explicit formulas for this, rewriting the generalized Chamber Ansatz of \cite{Fomin1999} in terms of the $X_i$.  We will see that the resulting formula recovers the exchange matrix defined in \cite{Berenstein2005}.  

\begin{prop}\label{prop:XtoA}
Fix a double reduced word $\mb{i}$ for $(u,v)$, let $\{X_i\}_{i \in I}$ be the corresponding rational coordinates on $G^{u,v}_{\Ad}$, and let $\{A_i\}_{i \in I}$ be the corresponding generalized minors on $G^{u,v}$.  Then if $p_G:G \to G_{\Ad}$ is the composition of the automorphism $\io \circ \ze^{u,v}$ of $G^{u,v}$ with the quotient map $G \to G_{\Ad}$, we have
\[
p_G^*(X_j) = \prod_{k \in I}A_k^{\wt{B}_{j,k}}.
\]
Here $\wt{B} = B+M$, where $B$ and $M$ are the $I \times I$ matrices given by\footnote{We keep the notation introduced at the beginning of \cref{sec:dbcfact}.}
\begin{align*}
B_{jk} & = \frac{C_{|i_k|,|i_j|}}{2}\biggl(\ep_j[j=k^+] - \ep_k[j^+=k] + \ep_j[k<j<k^+][j>0] - \ep_{j^+}[k<j^+<k^+][j^+\leq m]  \\
& \quad - \ep_k[j<k<j^+][k>0] + \ep_{k^+}[j<k^+<j^+][k^+\leq m] \biggr)
\end{align*}
and
\begin{align*}
M_{jk} &  = \frac12 C_{|i_k|,|i_j|}\biggl( [j^+,k^+>m] + [j,k <0] \biggr).
\end{align*}

\end{prop}

\begin{proof}
Recall from \cref{prop:fincov} that the image of $T_{\mb{i}}$ in $G^{u,v}$ is a finite cover of $\CX_{\mb{i}}$ in $G^{u,v}_{\Ad}$ under the quotient map.  Thus it follows from \cref{thm:main} that there exists some integer matrix $N$ such that
\[
p_G^*(X_j) = \prod_{k \in I}A_k^{N_{jk}}.
\]

To compute $N$, define new variables $t'_1,\dots,t'_{m+\wt{r}}$ by
\[
t'_k =  \prod_{\substack{j < k \\ |i_j| = |i_k|}}X_j^{\ep_k}.
\]
Here if $k>m$ we set $|i_k| = k-m$ and $\ep_k = +1$.  The $t'_k$ are uniquely determined by the requirement that
\[
X^{\om_{\wt{r}}^\vee}_{-\wt{r}}\cdots X^{\om_{1}^\vee}_{-1} E_{i_1} X^{\om_{|i_1|}^\vee}_1 \cdots E_{i_m} X^{\om_{|i_m|}^\vee}_m = x_{i_1}(t'_1) \cdots x_{i_m}(t'_m)\prod_{k=1}^{\wt{r}}(t'_{m+k})^{\om_k^\vee}.
\]

Moreover, inverting this change of variables one finds that
\begin{align}\label{eq:Xtot'}
X_j = \prod_{1 \leq k \leq m+\wt{r}}(t'_k)^{D_{jk}},
\end{align}
where $D$ is the integer matrix with rows labelled by $I$, columns labelled by $1,\dots,m+\wt{r}$, and
\[
D_{jk} = ([j^+ = k] - [j = k])\ep_k.
\]

We now compare the $t'_k$ with the coordinates $t_k$ on $G^{u,v}$ induced from
\[
x_{\mb{i}}: (t_1,\dots,t_{m+\wt{r}}) \mapsto x_{i_1}(t_1)\cdots x_{i_m}(t_m)\prod_{k=1}^{\wt{r}}(t_{m+k})^{\al_k^\vee}.
\]
If $\pi_G:G^{u,v} \to G_{\Ad}^{u,v}$ is the quotient map, then we can check that
\begin{align}\label{eq:t'tot}
\pi_G^*t'_j = \prod_{k=1}^{m+\wt{r}} t_k^{E_{jk}},
\end{align}
where $E$ is the $(m+\wt{r}) \times (m+\wt{r})$ matrix given by
\[
E_{jk} = \de_{jk}[j\leq m] + C_{|i_k|,|i_j|}[j,k > m].
\]

By \cref{thm:main} we have
\begin{align}\label{eq:ttoA}
(\io \circ \ze^{u,v})^*t_j = \prod_{k \in I}A_k^{F_{j,k}},
\end{align}
where $F_{j,k}$ is the integer matrix with rows labelled by $1,\dots,m+\wt{r}$, columns labelled by $I$, and
\[
F_{jk} = [j \leq m]\Psi_{j,k} + \frac12 [j >m][|i_j| = |i_k|]( \ep_{k^+} - \ep_k ).
\]
Here $\Psi_{j,k}$ is as in \cref{sec:dbcfact}, and if $k_+ >m$ for some $k \in I$, we set $\ep_{k^+} = +1$.  

We can now compute $N$ by multiplying the matrices $D$, $E$, and $F$, and simplifying the resulting conditional expression.  Before doing any serious simplification, a straightforward initial calculation yields
\begin{align}\label{eq:Linitial}
N_{jk} = [j^+ \leq m]\ep_{j^+}\Psi_{j^+,k} - [j > 0]\ep_j\Psi_{j,k} + \frac{C_{|i_k|,|i_j|}}{2}[j^+ > m](\ep_{k^+} - \ep_k).
\end{align}
Unwinding the definition of $\Psi$ we see that
\begin{align*}
\ep_j\Psi_{j,k} & = \frac{C_{|i_k|,|i_j|}}{2}\biggl( -\ep_k[j = k] - \ep_k[j = k^+] - (\ep_{j} + \ep_k)[k < j < k^+] + (\ep_{k^+} - \ep_k)[k^+ < j] \biggr).
\end{align*}
Plugging this and the corresponding expression for $\ep_{j^+}\Psi_{j^+,k}$ into \cref{eq:Linitial}, we obtain
\begin{equation}\label{eq:Lsecond}
\begin{split}
N_{jk} & = \frac{C_{|i_k|,|i_j|}}{2}\biggl( \ep_k[j=k]\bigl([j > 0] - [j^+ \leq m]\bigr) - \ep_k[j^+ = k] + \ep_k[j = k^+] \\
& \quad + (\ep_j + \ep_k)[k < j <k^+][j>0] - (\ep_{j^+} + \ep_k)[k<j^+<k^+][j^+ \leq m] \\
& \quad + (\ep_{k^+} - \ep_k)\bigl( [k^+<j^+][j^+ \leq m] - [k^+<j][j>0] + [j^+>m]\bigr) \biggr).
\end{split}
\end{equation}
The reader may verify that for any $j,k \in I$,
\begin{multline*}
[k^+<j^+][j^+ \leq m] - [k^+<j][j>0] + [j^+>m] \\
= [j<k^+<j^+][k^+\leq m] + [j=k^+] + [j^+,k^+>m].
\end{multline*}
This identity lets us rewrite \cref{eq:Lsecond} as
\begin{equation}\label{eq:Lthird}
\begin{split}
N_{jk} & = \frac{C_{|i_k|,|i_j|}}{2}\biggl( [j=k]\bigl( [j<0] + [j^+>m]\bigr) + \ep_j[j = k^+] - \ep_k[j^+=k] \\
& \quad + (1-\ep_k)[j^+,k^+>m][j \neq k] + \ep_j[k<j<k^+][j>0]  \\
& \quad - \ep_{j^+}[k<j^+<k^+][j^+\leq ] + \ep_{k^+}[j<k^+<j^+][k^+\leq m] \\
& \quad + \ep_k\bigl([k<j<k^+][j>0] - [k<j^+<k^+][j^+\leq m] \\
& \quad - [j<k^+<j^+][k^+\leq m]\bigr) \biggr).
\end{split}
\end{equation}
By another boolean computation the reader may check that
\begin{multline*}
[k<j<k^+][j>0] - [k<j^+<k^+][j^+\leq m] - [j<k^+<j^+][k^+\leq m] \\
 = -[j<k<j^+][k>0] + [j\neq k]\bigl([j^+,k^+>m] - [j,k<0]\bigr)
\end{multline*}
for any $j,k \in I$.
But now we can use this to rewrite \cref{eq:Lthird} as
\begin{align*}
N_{jk} & = \frac{C_{|i_k|,|i_j|}}{2}\biggl( [j^+,k^+>m] + [j,k<0] + \ep_j[j = k^+] - \ep_k[j^+=k]  + \ep_j[k<j<k^+][j>0] \\
& \quad  - \ep_{j^+}[k<j^+<k^+][j^+\leq ] - \ep_k[j<k<j^+][k>0] + \ep_{k^+}[j<k^+<j^+][k^+\leq m] \biggr) \\
& = \wt{B}_{j,k},
\end{align*}
completing the proof.
\end{proof}

\section{Cluster Algebras and Double Bruhat Cells}\label{sec:clusterdbc}
Corresponding to a double reduced word for $(u,v)$ we associated in \cref{sec:words} a collection of generalized minors.  In \cite{Fomin1999} it was discovered that as the double reduced word is varied, these collections vary by certain subtraction-free relations, which served as prototypes for the cluster algebra exchange relations introduced in \cite{Fomin2001}.  In \cite{Berenstein2005} it was shown that the generalized minors are organized into an upper cluster algebra structure on the coordinate ring of a double Bruhat cell in a semisimple algebraic group; in this section we extend this result to the double Bruhat cells of any symmetrizable Kac-Moody group.  

In fact, the cluster algebra associated with a double Bruhat cell is encoded by an exchange matrix we have already seen, when we computed the inverse of the coweight parametrization in \cref{sec:XandA}.  This is an instance of a general phenomenon, that one can define $\CX$-coordinates from cluster variables via the monomial transformation defined by the exchange matrix.  In the present situation, however, this is reversed: we start with independently defined cluster variables and $\CX$-coordinates, and derive this monomial transformation directly from the Chamber Ansatz.  We summarize our main results in \cref{thm:ensthm}, which relates the Chamber Ansatz and the simply-connected and adjoint forms of the double Bruhat cell as components of a nondegenerate cluster ensemble in the sense of \cite{Fock2003}.

\subsection{Cluster Algebras and $\CX$-coordinates}\label{sec:clusters}
Cluster algebras are commutative rings equip\-ped with a collection of distinguished generating sets related by an iterative process of mutation.  The same combinatorial data giving rise to the dynamics of mutation encodes a second algebraic structure, variously called $\tau$-coordinates \cite{Gekhtman2002}, coefficients or $Y$-variables \cite{Fomin2006}, and $\CX$-coordinates \cite{Fock2003}.  The cluster variables and $\CX$-coordinates are related by a canonical monomial transformation, which we refer to as the cluster ensemble map following \cite{Fock2003}.  We briefly recall the details we will need below; for in-depth discussions the reader may consult \cite{Fock2003,Fomin2006}.  We more or less follow \cite{Fock2003}, though we adapt our notation to be consistent with \cite{Berenstein2005} where possible.  

Cluster algebras and $\CX$-coordinates are defined by seeds.  A seed $\Si = (I,I_0,B,d)$ consists of the following data: 
\begin{enumerate}
\item An index set $I$ with a subset $I_0 \subset I$ of ``frozen'' indices. 
\item A rational $I \times I$ \emph{exchange matrix} $B$.  It should have the property that $b_{ij} \in \BZ$ unless both $i$ and $j$ are frozen.  
\item A set $d = \{d_i \}_{i \in I}$ of positive integers that skew-symmetrize $B$; that is, $b_{ij}d_j = -b_{ji}d_i$ for all $i,j \in I$.\footnote{The convention here and in \cite{Fock2006} is to denote by $d_i$ what is labelled $d_i^{-1}$ in \cite{Fock2003}.  The specific choice of $d_i$ only serves to specify a particular scaling of the Poisson bracket on the $\CX$-coordinates, hence many versions of this definition only ask that $B$ be skew-symmetrizable without taking the choice of $d_i$ as part of the data.}
\end{enumerate}

To construct a cluster algebra from a seed one considers the collection of all seeds obtained from an initial seed by mutation.  Let $k \in I \setminus I_0$ be an unfrozen index of a seed $\Si$.  We say another seed $\Si' = \mu_k(\Si)$ is obtained from $\Si$ by mutation at $k$ if we identify the index sets in such a way that the frozen variables and $d_i$ are preserved, and the exchange matrix $B'$ of $\Si'$ satisfies
\begin{align}\label{eq:matmut}
b'_{ij} = \begin{cases}
-b_{ij} & i = k \text{ or } j=k \\
b_{ij} & b_{ik}b_{kj} \leq 0 \\
b_{ij} + |b_{ik}|b_{kj} & b_{ik}b_{kj} > 0.
\end{cases}
\end{align}
Two seeds $\Si$ and $\Si'$ are said to be mutation equivalent if they are related by a finite sequence of mutations.

To a seed $\Si$ we associate a collection of \emph{cluster variables} $\{A_i\}_{i \in I}$ and a split algebraic torus $\CA_\Si := \Spec \BZ[A^{\pm 1}_I]$, where $\BZ[A^{\pm 1}_I]$ denotes the ring of Laurent polynomials in the cluster variables.  If $\Si'$ is obtained from $\Si$ by mutation at $k \in I \setminus I_0$, there is a birational \emph{cluster transformation} $\mu_k: \CA_\Si \to \CA_{\Si'}$.  This is defined by the \emph{exchange relation}
\begin{align}\label{eq:Atrans}
\mu_k^*(A'_i) = \begin{cases}
A_i & i \neq k \\
A_k^{-1}\biggl(\prod_{b_{kj}>0}A_j^{b_{kj}} + \prod_{b_{kj}<0}A_j^{-b_{kj}}\biggr) & i = k.
\end{cases}
\end{align}
These transformations provide gluing data between any tori $\CA_\Si$ and $\CA_{\Si'}$ of mutation equivalent seeds $\Si$ and $\Si'$.  The $\CA$-space $\CA_{|\Si|}$ is defined as the scheme obtained from gluing together all such tori of seeds mutation equivalent with an initial seed $\Si$.  

\begin{defn}
Let $\Si$ be a seed.  The cluster algebra $\CA(\Si)$ is the $\BZ$-subalgebra of the function field of $\CA_{|\Si|}$ generated by the collection of all cluster variables of seeds mutation equivalent to $\Si$.  The upper cluster algebra $\ol{\CA}(\Si)$ is
\[
\ol{\CA}(\Si) := \BZ[\CA_{|\Si|}] = \bigcap_{\Si' \sim \Si} \BZ[\CA_{\Si'}] \subset \BQ(\CA_{|\Si|}),
\]
or the intersection of all Laurent polynomial rings in the cluster variables of seeds mutation equivalent to $\Si$.  
\end{defn}

This definition is equivalent to that of \cite{Berenstein2005}, though the details appear somewhat different.  In particular, the exchange relations among cluster variables only involve the submatrix formed by the unfrozen rows of $B$, and in \cite{Berenstein2005} the term exchange matrix refers to (the transpose of) this submatrix.  Furthermore, it is implicit in our formulation that we will only consider cluster algebras of geometric type.

A key property of cluster algebras is the Laurent phenomenon, summarized in the following proposition.

\begin{prop}
\emph{(\cite[3.1]{Fomin2001})} For any seed $\Si$ the cluster algebra $\CA(\Si)$ is contained in the upper cluster algebra $\ol{\CA}(\Si)$.  In other words, the cluster variables of any seed are Laurent polynomials in the cluster variables of any seed mutation equivalent to it.
\end{prop}

A generic seed is mutation equivalent to infinitely many other seeds.  However, the following proposition guarantees that in favorable circumstances an upper cluster algebra is already determined by a finite number of them.  

\begin{prop}\label{prop:uppercluster}
\emph{(\cite[1.9]{Berenstein2005})} Let $\Si$ be a seed such that the submatrix of $B$ formed by its unfrozen rows has full rank.  Then 
\[
\ol{\CA}(\Si) = \BZ[\CA_\Si] \cap \bigcap_{k \in I \setminus I_0} \BZ[\CA_{\mu_k(\Si)}].
\]
In other words, the upper cluster algebra $\ol{\CA}(\Si)$ only depends on $\Si$ and the seeds obtained from it by a single mutation.
\end{prop}

Given a seed $\Si$ we also associate a second algebraic torus $\CX_\Si := \Spec \BZ[X_I^{\pm 1}]$, where $\BZ[X_I^{\pm 1}]$ again denotes the Laurent polynomial ring in the variables $\{X_i\}_{i \in I}$.  If $\Si'$ is obtained from $\Si$ by mutation at $k \in I \setminus I_0$, we again have a birational map $\mu_k: \CX_\Si \to \CX_{\Si'}$.  It is defined by
\begin{align}\label{eq:Xtrans}
\mu_k^*(X'_i) = \begin{cases}
X_iX_k^{[b_{ik}]_+}(1+X_k)^{-b_{ik}} & i \neq k  \\
X_k^{-1} & i = k,
\end{cases}
\end{align}
where $[b_{ik}]_+:=\mathrm{max}(0,b_{ik})$.  The $\CX$-space $\CX_{|\Si|}$ is defined as the scheme obtained from gluing together all such tori of seeds mutation equivalent with an initial seed $\Si$.   The transformation rules \cref{eq:Xtrans} were also discovered in \cite{Gekhtman2002}, and coincide with the transformation rules of coefficients \cite{Fomin2006}.  

Since $B$ is skew-symmetrizable, there is a canonical Poisson structure on each $\CX_\Si$ given by
\[
\{X_i,X_j\} = b_{ij}d_jX_iX_j.
\]
The cluster transformations of \cref{eq:Xtrans} intertwine the Poisson brackets on $\CX_{\Si}$ and $\CX_{\Si'}$, hence these assemble into a Poisson structure on $\CX_{|\Si|}$.  

Although in general the $\CA$- and $\CX$-spaces associated with a seed are defined over $\BZ$, we will only consider the associated complex schemes in the remainder of the paper.  In fact, since the expressions in \cref{eq:Atrans,eq:Xtrans} are subtraction-free, one can consider the associated $\BP$-points of these spaces for any semifield $\BP$.  This leads in particular to the notion of the positive real part of these spaces, but this will not play a direct role in the present work.

The exchange matrix encodes not only the structure of cluster transformations, but also a certain transformation between the two types of coordinates.   Concrete instances of this include the projection from decorated Teichm\"{u}ller space to Teichm\"{u}ller space \cite{Fock2005} and the transformation of $T$-system solutions into solutions of the corresponding $Y$-system \cite{Kuniba2011}.  It was first defined abstractly in the study of compatible Poisson structures on a cluster algebra \cite{Gekhtman2002}, and in \cite{Fomin2006} played a key role in the derivation of universal formulas for cluster variables in terms of $F$-polynomials.  Following the terminology of \cite{Fock2003} we refer to it as the cluster ensemble map.  More precisely, we will need a slight generalization described in the following proposition, motivated by the formula found in \cref{prop:XtoA}.  

\begin{prop}\label{prop:modifiedensmap}
Let $M$ be an $I \times I$ matrix such that $M_{ij} = 0$ unless both $i$ and $j$ are frozen.  Let $\Si$ be any seed such that $\wt{B}=B+M$ is an integer matrix, and let $p_M: \CA_\Si \to \CX_\Si$ be the regular map defined by
\[
p_M^*(X_i) = \prod_{j \in I}A_j^{\wt{B}_{ij}}.
\]
Then $p_M$ extends to a regular map $p_M: \CA_{|\Si|} \to \CX_{|\Si|}$.\footnote{A special case of this is proved in \cite[Lemma 1.3]{Gekhtman2002}.}
\end{prop}

\begin{proof}
First observe that if $\Si'$ is any seed mutation equivalent to $\Si$, its exchange matrix $B'$ again has the property that $B' + M$ has integer entries.  This follows from the fact that the mutation rules \cref{eq:matmut} can only change the exchange matrix entries by integer values.  In particular, the formula in the statement of the proposition yields a regular map $p'_M: \CA_{\Si'} \to \CX_{\Si'}$ when we replace $B$ by $B'$.  

To check that these descend to a map $\CA_{|\Si|} \to \CX_{|\Si|}$, we must verify that they commute with the cluster transformations.  That is, if $\Si'$ is obtained from $\Si$ by mutation at $k$, we want to show that that there is a commutative diagram

%\centerline{\includegraphics{CEKMimage1.pdf}}

\vspace{-2mm}
 \[
\begin{tikzcd}
\CA_{\Si} \arrow[dashed]{r}{\mu_k} \arrow{d}{p_M} & \CA_{\Si'} \arrow{d}{p'_M} \\
\CX_{\Si} \arrow[dashed]{r}{\mu_k} & \CX_{\Si'} \\
\end{tikzcd}
\]
Note that for the special case $M = 0$ this is the content of \cite[Proposition 2.2]{Fock2003}, and that in general $p_M^*(X_i) = p_0^*(X_i)\prod_{j \in I_0}A_j^{M_{ij}}$.  If $i \neq k$, we have
\begin{align*}
(\mu_k \circ p_M)^*(X'_i) & = \mu_k^*\bigl( p_0(X'_i)\prod_{j \in I_0}(A'_j)^{M_{ij}} \bigr) \\
& = (\mu_k \circ p_0)^*(X'_i)\prod_{j \in I_0}A_j^{M_{ij}}
\end{align*}
and
\begin{align*}
(p_M \circ \mu_k)^*(X'_i) & = p_M^*\bigl( X_iX_k^{[b_{ik}]_+}(1+X_k)^{-b_{ik}}\bigr) \\
& = (p_0 \circ \mu_k)^*(X'_i)\prod_{j \in I_0}A_j^{M_{ij}},
\end{align*}
and the equality of these follows from their equality in the $M = 0$ case.  On the other hand, since $p_0^*(X_k) = p_M^*(X_k)$, it follows trivially that $(\mu_k \circ p_M)^*(X'_k) = (p_M \circ \mu_k)^*(X'_k)$, and the proposition follows.
\end{proof}

\subsection{Seeds Associated with Double Reduced Words}\label{sec:wordseed}

Before reinterpreting the results of \cref{sec:dbc} in terms of cluster algebras, let us explain how to associate a seed $\Si_{\mb{i}}$ with any double reduced word $\mb{i}$ for $(u,v)$.  This allows us to state the main result, \cref{thm:ensthm}, which incorporates the generalized minors and twist map into a modified cluster ensemble in the sense of \cref{prop:modifiedensmap}.  

\begin{defn}\label{def:wordseed}
Let $\mb{i}$ be a double reduced word for $(u,v)$, and let $m = \ell(u) + \ell(v)$.  We define a seed $\Si_{\mb{i}}$ as follows.  The index set is $I = \{-\wt{r},\dots,-1\} \cup \{1,\dots,m\}$, and an index $k \in I$ is frozen if either $k<0$ or $k^+ > m$.  To each index $k>0$ is associated a weight $1 \leq |i_k| \leq \wt{r}$, which we extend to $k < 0$ by setting $|i_k| = |k|$.  The exchange matrix $B:=B_{\mb{i}}$ is defined by
\begin{align*}
b_{jk} & = \frac{C_{|i_k|,|i_j|}}{2}\biggl(\ep_j[j = k^+] - \ep_k[j^+=k]  \\
& \quad + \ep_j[k<j<k^+][j>0] - \ep_{j^+}[k<j^+<k^+][j^+\leq m] \\
& \quad - \ep_k[j<k<j^+][k>0] + \ep_{k^+}[j<k^+<j^+][k^+\leq m] \biggr).
\end{align*}
We let $d_k = d_{|i_k|}$, where the right-hand side refers to the symmetrizing factors of the Cartan matrix.  One easily checks that the skew-symmetrizability of $B$ follows from the symmetrizability of the Cartan matrix. 
\end{defn}

Note that the exchange matrix defined in \cite{Berenstein2005} is equal to the transpose of the matrix formed by the unfrozen rows of $B$.  Our main results are summarized in the following theorem.

\begin{thm}\label{thm:ensthm}
Let $G$ be a symmetrizable Kac-Moody group, $u,v \in W$ elements of its Weyl group, and $\mb{i}$ a double reduced word for $(u,v)$.  Consider the seed $\Si_{\mb{i}}$ defined in \cref{def:wordseed} and let $\CA_{|\Si_{\mb{i}}|}$, $\CX_{|\Si_{\mb{i}}|}$ be the associated complex $\CA$\-- and $\CX$-spaces.  Let $M$ be the $I \times I$ matrix with entries
\[
M_{jk} = \frac12 C_{|i_k|,|i_j|}\biggl( [j^+,k^+>m] + [j,k <0] \biggr),
\]
and let $p_G: G^{u,v} \to G_{\Ad}^{u,v}$ be the composition of the automorphism $\io \circ \ze^{u,v}$ of $G^{u,v}$ from \cref{thm:main} and the quotient map from $G$ to $G_{\Ad}$.  
\begin{enumerate}
\item
There is a regular map $a_{|\Si_{\mb{i}}|}: \CA_{|\Si_{\mb{i}}|} \to G^{u,v}$ which identifies the generalized minors of \cref{defn:Aminors} with the corresponding cluster variables on $\CA_{\Si_{\mb{i}}}$. It induces an isomorphism of $\BC[G^{u,v}]$ and the upper cluster algebra $\BC[\CA_{|\Si_{\mb{i}}|}]$.  
\item 
There is a regular map $x_{|\Si_{\mb{i}}|}: \CX_{|\Si_{\mb{i}}|} \to G_{\Ad}^{u,v}$ which extends the map $\CX_{\Si_{\mb{i}}} \to G_{\Ad}^{u,v}$ of \cref{defn:Xdef}.  It is Poisson with respect to the standard Poisson-Lie structure on $G_{\Ad}$ and the Poisson structure on $\CX_{|\Si_{\mb{i}}|}$ defined by the exchange matrix $B$.
\item
The matrix $\wt{B} = B + M$ has integer entries, hence there is an associated regular map $p_M:\CA_{|\Si_{\mb{i}}|} \to \CX_{|\Si_{\mb{i}}|}$.  These maps together form a commutative diagram:

%\centerline{\includegraphics{CEKMimage2.pdf}}

\vspace{-2mm}
\[
\begin{tikzcd}
\CA_{|\Si_{\mb{i}}|} \arrow{r}{a_{|\Si_{\mb{i}}|}} \arrow{d}{p_M} & G^{u,v} \arrow{d}{p_G} \\
\CX_{|\Si_{\mb{i}}|} \arrow{r}{x_{|\Si_{\mb{i}}|}} & G_{\Ad}^{u,v}.
\end{tikzcd}
\]
\end{enumerate} 
\end{thm}

The proof will occupy the rest of the paper.  We treat each statement separately, as \cref{thm:clusterthm,prop:mainX,prop:Poissonmap}.

\begin{remark}
The term cluster ensemble was used in \cite{Fock2003} to refer to the complete structure formed by the pair $\CA_{|\Si|}$, $\CX_{|\Si|}$ and the map $p_0$.  In general $p_0$ has positive dimensional fibers, and its image is a symplectic leaf of $G_{\Ad}^{u,v}$.  However, it is clear from \cref{prop:XtoA} that $p_M$ is a finite covering map.  Thus it is natural to summarize \cref{thm:ensthm} as saying that the double Bruhat cells $G^{u,v}$, $G_{\Ad}^{u,v}$ and the map $p_G$ form a ``nondegenerate'' cluster ensemble.  

This statement should be understood with the caveat that the maps $a_{|\Si_{\mb{i}}|}$, $x_{|\Si_{\mb{i}}|}$ are typically not biregular; rather, the complement of their images will have codimension at least 2.  In addition, the scheme $\CX_{|\Si|}$ is not separated in general.  Thus while the restriction of $x_{|\Si_{\mb{i}}|}$ to any individual torus $\CX_{\Si}$ is injective, this is not obviously the case for the entire map $x_{|\Si_{\mb{i}}|}$.
\end{remark}

\begin{example}
The exact form of the modified exchange matrix $\wt{B}$ is clarified by considering the degenerate example where $u$ and $v$ are the identity.  The relevant double Bruhat cells are then the Cartan subgroups $H$ and $H_{\Ad}$, and the cluster variables and $\CX$-coordinates are their respective coroot and coweight coordinates.  The change of variables between these is the Cartan matrix, and this is exactly what the definition of $\wt{B}$ reduces to in this case (note that the twist map is trivial when $u$ and $v$ are).  

The theorem then says that in general to get the twisted change of variables matrix, we add to the exchange matrix a copy of the Cartan matrix split in half between the ``left'' and ``right'' frozen variables.  As a typical example, let $u$ and $v$ be Coxeter elements of the affine group of type $A_1^{(1)}$.  For the natural choice of fundamental weights the extended Cartan matrix is 
\[
C = \begin{pmatrix} 2 & -2 & 1 \\ -2 & 2 & 0 \\ 1 & 0 & 0 \end{pmatrix}.
\]
\end{example}
If we take $\mb{i} = (-1,-2,1,2)$, then from the definitions one checks that
\begin{gather*}
B = \begin{pmatrix} 0 & 0 & -\frac12 & 1 & 0 & -\frac12 & 0 \\ 0 & 0 & 1 & -2 & 1 & 0 & 0 \\ \frac12 & -1 & 0 & 1 & 0 & 0 & 0 \\ -1 & 2 & -1 & 0 & 0 & -1 & 0 \\ 0 & -1 & 0 & 0 & 0 & 2 & -1 \\ \frac12 & 0 & 0 & 1 & -2 & 0 & 1 \\ 0 & 0 & 0 & 0 & 1 & -1 & 0 \end{pmatrix}, \quad
M = \begin{pmatrix} 0 & 0 & \frac12 & 0 & 0 & \frac12 & 0 \\ 0 & 1 & -1 & 0 & 0 & 0 & 0 \\ \frac12 & -1 & 1 & 0 & 0 & 0 & 0 \\ 0 & 0 & 0 & 0 & 0 & 0 & 0 \\ 0 & 0 & 0 & 0 & 0 & 0 & 0 \\ \frac12 & 0 & 0 & 0 & 0 & 1 & -1 \\ 0 & 0 & 0 & 0 & 0 & -1 & 1 \end{pmatrix}, \\
\wt{B} = \begin{pmatrix} 0 & 0 & 0 & 1 & 0 & 0 & 0 \\ 0 & 1 & 0 & -2 & 1 & 0 & 0 \\ 1 & -2 & 1 & 1 & 0 & 0 & 0 \\ -1 & 2 & -1 & 0 & 0 & -1 & 0 \\ 0 & -1 & 0 & 0 & 0 & 2 & -1 \\ 1 & 0 & 0 & 1 & -2 & 1 & 0 \\ 0 & 0 & 0 & 0 & 1 & -2 & 1 \end{pmatrix}.
\end{gather*}
Note in particular that while $B$ is degenerate, reflecting the fact that the symplectic leaves of $G_{\Ad}^{u,v}$ have positive codimension, $|\det{\wt{B}}| = 2$, reflecting the fact that $p_G$ is a double cover.  Furthermore, $\wt{B}$ has integral entries, while $B$ may in general have half-integral entries where both the row and column correspond to frozen variables.

\subsection{Cluster Transformations of $\CX$-coordinates}\label{sec:Xtransformations}

Recall that in \cref{defn:Xdef} we constructed an explicit regular map $x_{\Si_{\mb{i}}}: \CX_{\Si_\mb{i}} \to G_{\Ad}^{u,v}$ (from now on we identify the tori $\CX_{\mb{i}}$ and $\CX_{\Si_{\mb{i}}}$ in the obvious way).  If $\Si'$ is obtained from $\Si_{\mb{i}}$ by a single mutation, we now show that this extends to a regular map $\CX_{\Si'} \to G_{\Ad}^{u,v}$, compatible with the cluster transformation between $\CX_{\Si_{\mb{i}}}$ and $\CX_{\Si'}$.  This generalizes a closely related statement in \cite[4.4]{Zelevinsky2000}.

\begin{prop}\label{prop:x_k}
Let $\Si_{\mb{i}}$ be the seed associated with a double reduced word $\mb{i}$, and $\CX_k:= \CX_{\mu_k(\Si_{\mb{i}})}$ for some index $k \in I \setminus I_0$.  There is a unique regular map $x_k:\CX_k \to G_{\Ad}^{u,v}$ such that the following diagram commutes:

%\centerline{\includegraphics{CEKMimage3.pdf}}

\vspace{-2mm}
\[
\begin{tikzcd}
\CX_{\Si_{\mb{i}}} \arrow[dashed]{rr}{\mu_k} \arrow{dr}[swap]{x_{\Si_{\mb{i}}}} && \CX_k \arrow{dl}{x_k} \\
& G_{\Ad}^{u,v} &
\end{tikzcd}
\]
\end{prop}

\begin{proof}
First note that since $\mu_k$ and $x_{\Si_{\mb{i}}}$ are birational, there is a unique rational map $x_k$ making the diagram commute; the claim is that this is in fact regular.  

We will let $Y_i:=X'_i$ denote the $\CX$-coordinates on $\CX_k$.  The cluster transformation \cref{eq:Xtrans} lets us express the $X_i$ as rational functions of the $Y_i$, and with this in mind we write the rational map $x_k$ as
\begin{align}\label{eq:mutXcoords}
(Y_{-\wt{r}},\dots,Y_m) \mapsto X^{\om_{\wt{r}}^\vee}_{-\wt{r}}\cdots X^{\om_1^\vee}_{-1} E_{i_1} X^{\om_{|i_1|}^\vee}_1 \cdots X^{\om_{|i_m|}^\vee}_m
\end{align}
Note that if $i > k^+$ or $i^+ < k$, we have $Y_i = X_i$ by \cref{eq:Xtrans} and \cref{def:wordseed}.  In particular, the corresponding terms in \cref{eq:mutXcoords} do not affect whether or not the overall expression defines a regular map.  Thus it suffices to consider the case where $k = 1$ and $k^+ = m$, to which we will now restrict our attention (given this, we will write $i$ in place of $|i_1|=|i_m|$).  

Define rational maps $g_j: \CX_1 \to G$ by
\begin{gather*}
g_j = \begin{cases} \biggl( \prod_{j \in I}X^{\om_{|i_j|}^\vee}_j \biggr) x_{i_1}(X_1^{-\ep_1}X_m^{-\ep_1})x_{i_m}(X_m^{-\ep_m}) & j = 1 \\ x_{i_m}(-X_m^{-\ep_m})x_{i_j}\biggl( \prod_{\substack{j \leq \ell <m \\ |i_j| = |i_\ell|}}X_{\ell}^{-\ep_j} \biggr) x_{i_m}(X_m^{-\ep_m})& 1 < j \leq m, \end{cases}
\end{gather*}
again interpreting the $X_i$ as rational functions of the $Y_i$ on the right-hand side.  
Then
\begin{gather*}
X^{\om_{\wt{r}}^\vee}_{-\wt{r}}\cdots X^{\om_1^\vee}_{-1} E_{i_1} X^{\om_{|i_1|}^\vee}_1 \cdots X^{\om_{|i_m|}^\vee}_m = g_1\cdots g_m,
\end{gather*}
so it suffices to prove that each $g_j$ is regular (and that their product lands in $G_{\Ad}^{u,v}$).  The details of the argument depend on the signs of $i_1$ and $i_m$, so we consider the distinct cases separately.

\textbf{Case 1, $i_1 = i_m = i$:} First consider $g_1$.  By \cref{def:wordseed} we have $b_{-i,1}=-1$ and $b_{m,1}=1$, hence
\[
X_{-i} = Y_{-i}Y_1(1+Y_1)^{-1}, \quad X_m = Y_m(1+Y_1).
\]
 Thus  
\begin{align*}
\biggl( \prod_{\substack{j \in I \\ |i_j| = i}}X_j^{\om_i^\vee} \biggr) & = \biggl( Y_{-i}Y_1(1+Y_1)^{-1} \biggr)^{\om_i^\vee} Y_1^{-\om_i^\vee} \biggl( Y_m(1+Y_1) \biggr)^{\om_i^\vee} \\
& = (Y_{-i}Y_m)^{\om_i^\vee},
\end{align*}
which is a regular function of the $Y_j$.  

In fact, for any $1 \leq j \leq \wt{r}$ such that $i \neq j$, there are as many indices $k \in I$ with $|i_k| = j$ and $b_{k,1} > 0$ as there are with $|i_k| = j$ and $b_{k,1} < 0$.  One has $b_{k,1} > 0$ exactly either when $1<k<k^+< m$ and $\ep_k = -\ep_{k^+} = -1$, or when $k=-j$, $1<k^+< m$, and $\ep_{k^+} = 1$.  Similarly $b_{k,1} < 0$ exactly either when $1<k<k^+<m$ and $\ep_k = -\ep_{k^+} = 1$, or when $1<k< m<k^+$ and $\ep_k = 1$. One can check that the latter situations are in bijection with the former.

If $|i_k| = j$ for some index $k \in I$, we have
\[
X_k = \begin{cases}
Y_k(1+Y_1)^{-C_{ij}} & b_{k,1} > 0 \\
Y_kY_1^{-C_{ij}}(1+Y_1)^{C_{ij}} & b_{k,1} <0 \\
Y_k & b_{k,1} = 0.
\end{cases}
\]
But then by the above remark the positive and negative powers of $(1+Y_1)$ in 
\[
\prod_{\substack{k \in I \\ |i_k| = j}}X^{\om_{j}^\vee}_k
\]
cancel each another out, leaving a total expression which depends regularly on the $Y_k$.  Since this holds for all $1\leq j \leq \wt{r}$, it follows that $\prod_{j \in I}X^{\om_{|i_j|}^\vee}_j$ is a regular function of the $Y_k$.  Furthermore, we have 
\begin{align*}
x_{i_1}(X_1^{-\ep_1}X_m^{-\ep_1})x_{i_m}(-X_m^{-\ep_m}) & = x_i\big( Y_1 Y_m^{-1}(1+Y_1)^{-1}\big) x_i\big( Y_m^{-1}(1+Y_1)^{-1}\big) \\
& = x_i(Y_m^{-1}),
\end{align*}
and it follows that $g_1$ is regular.

Now consider $g_j$ for $j>1$.  If $\ep_j = -1$, then by following a similar analysis as above one sees that $\prod_{\substack{j \leq \ell <m \\ |i_j| = |i_\ell|}}X_{\ell}^{-\ep_j}$ is actually a regular function of the $Y_k$, since all $(1+Y_1)$ terms cancel out. Since in this case the $E_i$ terms commute with $E_{i_j}$, it follows that $g_j$ is regular. 

If $\ep_j = 1$, then $\prod_{\substack{j \leq \ell <m \\ |i_j| = |i_\ell|}}X_{\ell}^{-\ep_j}$ is equal to $(1+Y_1)^{-C_{i,|i_j|}}$ times some Laurent monomial $q$ in the $Y_k$.  But then 
\[
x_i\big(-Y_m^{-1}(1+Y_1)^{-1}\big)x_{i_j}\big(q(1+Y_1)^{-C_{i,|i_j|}}\big)x_i\big(Y_m^{-1}(1+Y_1)^{-1}\big)
\]
is regular by \cref{lem:regularconj}.

\textbf{Case 2, $i_1 = i, i_m = -i$:}  Again, first consider $g_1$.  Now $b_{-i,1}$ and $b_{m,1}$ are both equal to $-1$, so 
\[
X_{-i} = Y_{-i}Y_1(1+Y_1)^{-1} \text{ and } X_m = Y_mY_1(1+Y_1)^{-1}.
\]
Thus 
\begin{align*}
\prod_{\substack{j \in I \\ |i_j| = i}}X^{\om_{|i_j|}^\vee}_j & = 
\biggl( Y_{-i}Y_1(1+Y_1)^{-1} \biggr)^{\om_{i}^\vee}Y_1^{-\om_i^\vee} \biggl( Y_mY_1(1+Y_1)^{-1} \biggr)^{\om_{i}^\vee} \\
& = \biggl( Y_{-i}Y_m(1+Y_1)^{-2} \biggr)^{\om_{i}^\vee}.
\end{align*}
This time for any  $1 \leq j \leq \wt{r}$ with $j \neq i$, there is exactly one more index $k \in I$ with $|i_k| = j$ and $b_{k,1} > 0$ than there is with $|i_k| = j$ and $b_{k,1} < 0$.  One has $b_{k,1} > 0$ exactly when either $1<k< m$ and $\ep_k = -1$, or $k=-j$ with either $k^+>m$ or $1<k^+< m$ and $\ep_{k^+} = 1$.  On the other hand $b_{k,1} < 0$ if and only if $1<k<k^+< m$ and $\ep_k = - \ep_{k^+} = 1$.  Thus 
\[
\prod_{\substack{k \in I \\ |i_k| = j}}X^{\om_{j}^\vee}_k
\]
 is the product of $(1+Y_1)^{-C_{ij}\om_j^\vee}$ and a term which is regular in the $Y_k$.  

It follows that $\prod_{j \in I}X^{\om_{|i_j|}^\vee}_j$ is the product of a regular term and
\begin{gather*}
\prod_{1 \leq j \leq \wt{r}}(1+Y_1)^{-C_{ij}\om_j^\vee} = (1+Y_1)^{-\al_i^\vee}.
\end{gather*}
Finally $g_1$ itself is then the product of a regular term and
\begin{align*}
(1+Y_1)^{-\al_i^\vee}x_{i_1}(X_1^{-\ep_1}X_m^{-\ep_1})x_{i_m}(X_m^{-\ep_m}) & = (1+Y_1)^{-\al_i^\vee}x_i\big(Y_m^{-1}(1+Y_1)\big)x_{-i}\big(Y_mY_1(1+Y_1)^{-1}\big) \\
& = \varphi_i\begin{pmatrix} 1 & Y_m^{-1} \\ Y_1Y_m & 1 + Y_1 \end{pmatrix},
\end{align*}
hence is regular.  

Now consider $g_j$ for $j>1$.  This time if $\ep_j = 1$, $\prod_{\substack{j \leq \ell <m \\ |i_j| = |i_\ell|}}X_{\ell}^{-\ep_j}$ is a Laurent monomial in the $Y_k$, the $(1+Y_1)$ terms cancelling.  If $\ep_j = -1$, the relevant expression becomes
\[
x_{-i}\big(-Y_mY_1(1+Y_1)^{-1}\big)x_{i_j}\big(q(1+Y_1)^{-C_{i,|i_j|}}\big)x_{-i}\big(Y_mY_1(1+Y_1)^{-1}\big)
\]
for some Laurent monomial $q$ in the $Y_k$.  Again, this is regular by \cref{lem:regularconj}.

The remaining cases of $i_1 = i_m = -i$ and $i_1 = -i_m = -i$ do not differ substantively from the above two; the details are left to the reader.  

It is clear that the image of $\CX_1$ in $G_{\Ad}$ lands in the closure of $G^{u,v}_{\Ad}$.  Consider the extension of the regular map $p_G:G^{u,v} \to G_{\Ad}^{u,v}$ to a rational map between their closures.  By \cref{prop:XtoA,prop:modifiedensmap} we can write the rational functions $p_G^*(Y_i)$ on $\ol{G^{u,v}}$ as Laurent monomials in $A'_1$ and the $A_i$ with $i \neq 1$, where $A'_1$ is the rational function on $G^{u,v}$ obtained by \cref{eq:Atrans}.  Since $p_G$ is a finite covering map, by \cref{prop:XtoA} the determinant $D$ of the matrix $\wt{B}$ is a nonzero integer.  In particular, we can write each $(A_i)^{D}$ with $i \neq 1$ as a Laurent monomial in the $p_G^*(Y_i)$.  But the generalized minors $\De_{u,e}^{\om_i}$ and $\De_{e,v^{-1}}^{\om_i}$ are frozen cluster variables, hence their $D$th powers can be expressed as Laurent monomials in the $p_G^*(Y_i)$.  Thus these powers, hence the minors themselves, are nonvanishing on $p_G^{-1}(\CX_1)$.  Since $p_G$ is the composition of a biregular automorphism of $G^{u,v}$ and the quotient map $\pi_G:G^{u,v} \to G^{u,v}_{\Ad}$, it follows that these minors do not vanish on $\pi_G^{-1}(\CX_1)$.  The fact that the image of $\CX_1$ lies in $G^{u,v}$ then follows by \cref{lem:closure}.  
\end{proof}

The following result was proved in finite type in \cite[Lemma 4.4]{Zelevinsky2000}.  However, the proof in loc.\ cited does not extend to the general case, as it involves exponentiating Lie algebra elements which in general have components in imaginary root spaces.

\begin{lemma}\label{lem:regularconj}
For distinct $1\leq i,j \leq r$ the map $\BC^* \times \BC \to N_\pm$ given by 
\[
(p,q) \mapsto x_{\pm i}(p^{-1})x_{\pm j}(p^{-C_{ij}}q)x_{\pm i}(-p^{-1})
\]
extends to a regular map $\BC^2 \to N_\pm$.
\end{lemma}
\begin{proof}
We prove the statement for $N_+$; the $N_-$ version then follows after applying the involution $\th$.  Recall from \cite[7.4]{Kumar2002} that the map
\[
N_+ \to \bigoplus_{1\leq i \leq \wt{r}}L(\om_i)^{\vee}, \quad n \mapsto n \cdot (v_1,\dots,v_{\wt{r}})
\]
is a closed embedding of ind-varieties, where $v_i$ is the lowest-weight vector of $L(\om_i)^{\vee}$.  Thus it suffices to show that
\[
(p,q) \mapsto x_i(p^{-1})x_j(p^{-C_{ij}}q)x_i(-p^{-1})\cdot v_k
\]
extends regularly to $p=0$ for all $1 \leq k \leq \wt{r}$. This is immediate unless $k$ is equal to $i$ or $j$.  

If $k=j$, then
\[
x_i(p^{-1})x_j(p^{-C_{ij}}q)x_i(-p^{-1})\cdot v_j = x_i(p^{-1})\cdot(v_j + p^{-C_{ij}}qe_jv_j),
\]
where $e_j$ is $j$th the positive Chevalley generator.  Since $e_jv_j$ is a lowest-weight vector for the $\varphi_i(SL_2)$-subrepresentation it generates and $\langle -\om_j + \al_j | \al_i^\vee \rangle = C_{ij}$, we have
\begin{align*}
x_i(p^{-1})\cdot(v_j + p^{-C_{ij}}qe_jv_j) & = \sum_{n=0}^{\infty}p^{-n}\frac{e_i^n}{n!}(v_j + p^{-C_{ij}}qe_jv_j) \\
& = v_j + \sum_{n=0}^{-C_{ij}}p^{-C_{ij} - n}\frac{qe_i^ne_j}{n!}v_j.
\end{align*}
Since this last expression depends only on nonnegative powers of $p$, the claim follows.  

If $k=i$, a similar calculation yields
\begin{align*}
x_i(p^{-1})x_j(p^{-C_{ij}}q)x_i(-p^{-1})\cdot v_i  & = x_i(p^{-1})x_j(p^{-C_{ij}}q)\cdot(v_i - p^{-1}e_iv_i)\\
& = x_i(p^{-1})\cdot\biggl( v_i - \sum_{n=0}^{-C_{ij}}p^{-1-nC_{ij}}\frac{q^ne_j^ne_i}{n!}v_i \biggr).
\end{align*}
If $n>0$, $e_j^ne_iv_i$ is a lowest-weight vector for the $\varphi_i(SL_2)$-subrepresentation it generates.  Otherwise, $-\om_i + n \al_j$ would have a nonzero weight space in $L(\om_i)^\vee$, which would generate a nontrivial $\varphi_j(SL_2)$-representation containing $v_i$, a contradiction.  

Since $\langle -\om_i +\al_i + n \al_j | \al_i^\vee \rangle = 1 + n C_{ij}$, 
\[
x_i(p^{-1})\cdot p^{-1-nC_{ij}}\frac{q^ne_j^ne_i}{n!}v_i = \sum_{m = 0}^{ - 1 - n C_{ij}}p^{-1-nC_{ij}-m}\frac{q^ne_i^me_j^ne_i}{m!n!}v_i.
\]
But since $-1 -nC_{ij}-m \geq 0$ for all $m\leq -1 - nC_{ij}$, the right hand side depends only on nonnegative powers of $p$.  But $x_i(p^{-1})x_j(p^{-C_{ij}}q)x_i(-p^{-1})\cdot v_i$ is a sum of such terms with $n>0$ and
\[
x_i(p^{-1})\cdot(v_i - p^{-1}e_iv_i) = v_i,
\]
hence extends to a regular map at $p=0$. 
\end{proof}

\begin{lemma}
The closure of $G^{u,v}$ in $G$ is
\[
\ol{G^{u,v}} = \bigsqcup_{\substack{u' \leq u \\ v' \leq v}}G^{u',v'},
\]
where we use the Bruhat order on $W$.  If $x \in \ol{G^{u,v}}$, then $x \in G^{u,v}$ if and only if $\De_{u,e}^{\om_i}(x) \neq 0$ and $\De_{e,v^{-1}}^{\om_i}(x) \neq 0$ for all $1 \leq i \leq \wt{r}$.\footnote{In finite type a stronger version of this is stated in \cite[Proposition 2.8]{Berenstein2005}, following from the proof of \cite[Proposition 3.3]{Fomin2000}.}\label{lem:closure}
\end{lemma}
\begin{proof}
The decomposition of $\ol{G^{u,v}}$ follows easily from the corresponding statement about Schubert varieties \cite[7.1]{Kumar2002}.  It is also clear from their definitions that the stated generalized minors do not vanish on $G^{u,v}$.  Thus we must show that if $x \in \ol{G^{u,v}}\setminus G^{u,v}$, one of the stated minors vanishes on it.

Suppose that $u' \leq u$ in the Bruhat order.  By definition, there exist positive real roots $\be_1,\dots,\be_k$ such that $u = u'r_1\cdots r_k$, where $r_j\in W$ is the reflection
\[
r_j: \la \mapsto \la - \langle \la | \be^\vee_j \rangle \be_j.
\] 
 Here $\be_j^\vee$ is the positive coroot associated with $\be_j$.  Moreover, these satisfy $\ell(u'r_1) < \ell(u'r_1r_2) < \cdots < \ell(u)$, which in particular implies that $u'r_1 \cdots r_{j-1}(\be_j) > 0$ for all $j$ \cite[1.3.13]{Kumar2002}. 

If $u' \leq u$, we claim that for each $\om_i$, 
\[
u'(\om_i)-u(\om_i) \in \bigoplus_{1 \leq j \leq r} \BN \al_j.
\]
For any $1 < j \leq r$ we have 
\[
u'r_1 \cdots r_{j-1}(\om_i) - u'r_1 \cdots r_j(\om_i) = \langle \om_i | \be_j^\vee \rangle u'r_1 \cdots r_{j-1}(\be_j).
\]
But then
\begin{align*}
u'(\om_i)-u(\om_i) & = \sum_{1 < j \leq r}\big( u'r_1 \cdots r_{j-1}(\om_i) - u'r_1 \cdots r_j(\om_i)\big) \\
& =\sum_{1 < j \leq r} \langle \om_i | \be_j^\vee \rangle u'r_1 \cdots r_{j-1}(\be_j),
\end{align*}
which is indeed a sum of positive roots with nonnegative coefficients.  Furthermore, if $u'$ is strictly less than $u$ in the Bruhat order, $u'(\om_i)-u(\om_i)$ must be nonzero for some $1 \leq i \leq r$.  But then for any $x \in B_+ u' B_+$, we have $\De_{u,e}^{\om_i}(x) = 0$.  A straightforward adaptation of this argument implies that for any $x \in B_- v' B_-$ with $v' < v$, $\De_{e,v^{-1}}^{\om_i}(x) = 0$ for some $1 \leq i \leq r$, and the lemma follows.
\end{proof}

\subsection{Cluster Transformations of Generalized Minors}\label{sec:clustersandminors}

Recall that to a double reduced word $\mb{i}$ we associated in \cref{defn:Aminors} a collection $\{A_i\}_{i \in I}$ of generalized minors.  In this section we identify these with the cluster variables corresponding to the seed $\Si_{\mb{i}}$ and study their cluster transformations.  

\begin{thm} \label{thm:clusterthm}
There is a regular map $a_{|\Si_{\mb{i}}|}: \CA_{|\Si_{\mb{i}}|} \to G^{u,v}$ which identifies the generalized minors of \cref{defn:Aminors} with the corresponding cluster variables on $\CA_{\Si_{\mb{i}}}$. This map induces an isomorphism of $\BC[G^{u,v}]$ and the upper cluster algebra $\BC[\CA_{|\Si_{\mb{i}}|}]$.  
\end{thm}

When $G$ is a semisimple algebraic group, this is the content of \cite[2.10]{Berenstein2005}.  As in loc.\ cited, the proof we give is modelled on that of a closely related result in \cite{Zelevinsky2000}, which treats the case of reduced double Bruhat cells.  Most of the work is delegated to a series of lemmas that take up the bulk of the section; first we show how these lemmas assemble into the proof of \cref{thm:clusterthm}.

\begin{proof}[Proof of \cref{thm:clusterthm}]
By \cref{lem:fullrank}, \cref{prop:uppercluster} applies to $\Si_{\mb{i}}$, hence
\[
\BC[\CA_{|\Si_{\mb{i}}|}] = \BC[\CA_{\Si_{\mb{i}}}] \cap \bigcap_{k \in I \setminus I_0} \BC[\CA_k].
\]
On the other hand, by \cref{lem:codim}, the maps $a_{\Si_{\mb{i}}}: \CA_{\Si_{\mb{i}}} \to G^{u,v}$, $a_k: \CA_k \to G^{u,v}$ induce an isomorphism
\[
\BC[G^{u,v}] \cong \BC[\CA_{\Si_{\mb{i}}}] \cap \bigcap_{k \in I\setminus I_0} \BC[\CA_k].
\]
Then since $G^{u,v}$ is an affine variety (\cref{prop:fdimdbc}), we have $G^{u,v} \cong \Spec \BC[\CA_{|\Si_{\mb{i}}|}]$.  But then $a_{|\Si_{\mb{i}}|}$ is just the canonical map $\CA_{|\Si_{\mb{i}}|} \to \Spec \BC[\CA_{|\Si_{\mb{i}}|}]$.
\end{proof}

\begin{lemma} \label{lem:fullrank}
The submatrix of $B$ formed by its unfrozen rows has full rank.
\end{lemma}
\begin{proof}
First let 
\[
I_+ = \{ k \in I : k^- \in I \setminus I_0 \}.
\]
We claim the submatrix of $B$ whose rows are those indexed by $I \setminus I_0$ and whose columns are indexed by $I_+$ is lower triangular with nonzero diagonal entries.  The diagonal entries are of the form $b_{k,k^+}$, hence equal to $\pm 1$ by \cref{def:wordseed}.  On the other hand if an entry $b_{k,\ell}$ of this submatrix lies above the diagonal then $\ell > k^+$.   Again, from the definition of $B$ we must have $b_{k,\ell} = 0$.  Thus this square submatrix has full rank, and it follows that the matrix formed by the unfrozen rows has full rank.
\end{proof}

\begin{lemma}
For each unfrozen index $k \in I$, let $A'_k$ be the rational function on $G^{u,v}$ obtained from the exchange relation
\[
A_k' = A_k^{-1}\biggl(\prod_{b_{kj}>0}A_j^{b_{kj}} + \prod_{b_{kj}<0}A_j^{-b_{kj}}\biggr).
\]
Then $A'_k$ is in fact regular.
\end{lemma}
\begin{proof}
It suffices to consider the case $k=1$, $k^+=m$, where we will in fact show that $A'_1$ is the restriction to $G^{u,v}$ of a strongly regular function on $G$. In the general case, consider the double reduced word $\mb{i}' = (i_k,\dots,i_{k^+})$.  Then one has
\[
A'_{k,\mb{i}}(x) = A'_{1,\mb{i}'}(\ol{u_{<k}}^{-1}x\ol{v_{>k^+}}),
\]
hence $A'_{k,\mb{i}}$ is the restriction of a strongly regular function if $A'_{1,\mb{i}'}$ is.

We obtain the following formulas for $A_1'$ depending on the signs of $i_1$ and $i_m$.  We will let $E_\pm = \{1<j<m|\ep_j = \pm 1 \}$, $J_\pm = \{|i_j| | 1\leq j < m, j_- < 0 \}$, and $i := |i_1| = |i_m|$.
\begin{description}
\item[\textbf{Case 1, $i_1 = i_m = i$}] \hfill \\
\[
A_1'\De_{e,s_i}^{\om_i} = \De_{e,v^{-1}}^{\om_i}\prod_{\substack{k \in E_+ \\ k^+ \nin E_+}}(\De_{u_{\leq k}, v_{>k}}^{\om_{|i_k|}})^{-C_{|i_k|,i}}
	+ \De_{e,e}^{\om_i}\prod_{\substack{k \in E_+ \\ k^- \nin E_+}}(\De_{u_{<k},v_{\geq k}}^{\om_{|i_k|}})^{-C_{|i_k|,i}}
\]
\item[\textbf{Case 2, $i_1 = i_m = -i$}] \hfill \\
\[
A_1'\De_{s_i,e}^{\om_i} = \De_{u,e}^{\om_i}\prod_{\substack{k \in E_- \\ k^- \nin E_-}}(\De_{u_{< k}, v_{>k}}^{\om_{|i_k|}})^{-C_{|i_k|,i}}
	+ \De_{e,e}^{\om_i}\prod_{\substack{k \in E_- \\ k^+ \nin E_-}}(\De_{u_{\leq k},v_{> k}}^{\om_{|i_k|}})^{-C_{|i_k|,i}}
\]
\item[\textbf{Case 3, $i_1 = i$, $i_m = -i$}] \hfill \\
\[
A_1'\De_{e,e}^{\om_i} = \De_{e,v^{-1}}^{\om_i}\De_{u,e}^{\om_i}\prod_{\substack{k \in E_+ \\ k^+ \in E_-}}(\De_{u_{\leq k}, v_{>k}}^{\om_{|i_k|}})^{-C_{|i_k|,i}}
	+ \Biggl(\prod_{\substack{k \in E_- \\ k^- \nin E_-}}(\De_{u_{\leq k},v_{> k}}^{\om_{|i_k|}})^{-C_{|i_k|,i}}\Biggr)	\Biggl(\prod_{j \in [1,\wt{r}] \setminus J_-}(\De_{e,v^{-1}}^{\om_j})^{-C_{ij}}\Biggr)
\]
\item[\textbf{Case 4, $i_1 = -i$, $i_m = i$}] \hfill \\
\[
A_1'\De_{s_i,s_i}^{\om_i} = \De_{e,s_i}^{\om_i}\De_{s_i,e}^{\om_i}\prod_{\substack{k \in E_- \\ k^+ \in E_+}}(\De_{u_{\leq k}, v_{>k}}^{\om_{|i_k|}})^{-C_{|i_k|,i}}
	+ \Biggl(\prod_{\substack{k \in E_+ \\ k^+ \nin E_+}}(\De_{u_{\leq k},v_{> k}}^{\om_{|i_k|}})^{-C_{|i_k|,i}}\Biggr)
	\Biggl(\prod_{j \in [1,\wt{r}] \setminus J_+}(\De_{e,v^{-1}}^{\om_j})^{-C_{ij}}\Biggr)
\]
\end{description}

We now impose the further assumption that $j < k$ for all $j \in E_+$, $k \in E_-$, before returning to the general case.  Letting $S_\pm = \{|i_k| : k \in E_\pm \} \subset [1,\wt{r}]$, we can then simplify the above formulas as:

\begin{description}
\item[\textbf{Case 1, $i_1 = i_m = i$}] \hfill \\
\[
A_1'\De_{e,s_i}^{\om_i} = \De_{e,v^{-1}}^{\om_i}\prod_{\ell \in S_+}(\De_{e,e}^{\om_\ell})^{-C_{\ell i}}
	+ \De_{e,e}^{\om_i}\prod_{\ell \in S_+}(\De_{e,v^{-1}}^{\om_\ell})^{-C_{\ell i}}
\]
\item[\textbf{Case 2, $i_1 = i_m = -i$}] \hfill \\
\[
A_1'\De_{s_i,e}^{\om_i} = \De_{u,e}^{\om_i}\prod_{\ell \in S_-}(\De_{e,e}^{\om_\ell})^{-C_{\ell i}}
	+ \De_{e,e}^{\om_i}\prod_{\ell \in S_-}(\De_{u,e}^{\om_\ell})^{-C_{\ell i}}
\]
\item[\textbf{Case 3, $i_1 = i$, $i_m = -i$}] \hfill \\
\[
A_1'\De_{e,e}^{\om_i} = \De_{e,v^{-1}}^{\om_i}\De_{u,e}^{\om_i}\prod_{\ell \in S_+ \cap S_-}(\De_{e,e}^{\om_\ell})^{-C_{\ell i}}
	+ \Biggl(\prod_{\ell \in S_-}(\De_{u,e}^{\om_\ell})^{-C_{\ell i}}\Biggr)
	\Biggl(\prod_{\ell \in ([1,\wt{r}] \setminus S_-) \cup S_+}(\De_{e,v^{-1}}^{\om_\ell})^{-C_{\ell i}}\Biggr)
\]
\item[\textbf{Case 4, $i_1 = -i$, $i_m = i$}] \hfill \\
\[
A_1'\De_{s_i,s_i}^{\om_i} = \De_{e,s_i}^{\om_i}\De_{s_i,e}^{\om_i}
	+ \Biggl(\prod_{\ell \in [1,\wt{r}] \setminus \{i\}}(\De_{e,e}^{\om_\ell})^{-C_{\ell i}}\Biggr)
\]
\end{description}
In each case, one can apply \cref{prop:gendetid} to deduce that $A'_1$ is indeed regular.  For example, in case 1, multiplying both sides of the above equation by
\[
\prod_{j \in [1,\wt{r}] \setminus (\{i\} \cup S_+)}(\De_{e,e}^{\om_j})^{-C_{ji}}
= \prod_{j \in [1,\wt{r}] \setminus (\{i\} \cup S_+)}(\De_{e,v^{-1}}^{\om_j})^{-C_{ji}}
\]
we obtain
\begin{align*}
& A_1'\De_{e,s_i}^{\om_i}\Biggl( \prod_{j \in [1,\wt{r}] \setminus (\{i\} \cup S_+)}(\De_{e,e}^{\om_j})^{-C_{ji}} \Biggr)  & \\
& \quad = \De_{e,v^{-1}}^{\om_i}\prod_{\ell \in [1,\wt{r}]\setminus\{i\}}(\De_{e,e}^{\om_\ell})^{-C_{\ell i}}
	+ \De_{e,e}^{\om_i}\prod_{\ell \in [1,\wt{r}]\setminus\{i\}}(\De_{e,v^{-1}}^{\om_\ell})^{-C_{\ell i}} & \\
& \quad	= \De_{e,v^{-1}}^{\om_i}(\De_{e,e}^{\om_i} \De_{s_i,s_i}^{\om_i} - \De_{e,s_i}^{\om_i} \De_{s_i,e}^{\om_i})
	+ \De_{e,e}^{\om_i}(\De_{e,s_i}^{\om_i} \De_{s_i,v^{-1}}^{\om_i} - \De_{s_i,s_i}^{\om_i} \De_{e,v^{-1}}^{\om_i}) & \\
& \quad	= \De_{e,s_i}^{\om_i}(\De_{e,e}^{\om_i} \De_{s_i,v^{-1}}^{\om_i} - \De_{s_i,e}^{\om_i} \De_{e,v^{-1}}^{\om_i}). & 
\end{align*}
By \cref{prop:prime}, $\De_{e,s_i}^{\om_i}$ is a prime element of $\BC[G]$ distinct from the $\De_{e,e}^{\om_j}$ for $j \neq i$, hence $\prod_{j \in [1,\wt{r}] \setminus (\{i\} \cup S_+)}(\De_{e,e}^{\om_j})^{-C_{ji}}$ must divide $(\De_{e,e}^{\om_i} \De_{s_i,v^{-1}}^{\om_i} - \De_{s_i,e}^{\om_i}\De_{e,v^{-1}}^{\om_i})$ in $\BC[G]$.  But then 
\[
A'_1 = (\De_{e,e}^{\om_i} \De_{s_i,v^{-1}}^{\om_i} - \De_{s_i,e}^{\om_i}\De_{e,v^{-1}}^{\om_i})/\Biggl(\prod_{j \in [1,\wt{r}] \setminus (\{i\} \cup S_+)}(\De_{e,e}^{\om_j})^{-C_{ji}}\Biggr)
\]
is indeed an element of $\BC[G]$.  We omit the remaining cases, which may be dealt with using the same strategy.

Now suppose $\mb{i}$ and $\mb{i}'$ are two double reduced word differing only in that $i_k = i'_{k+1} = j$ and $i_{k+1} = i'_k = -j'$ for some $1 \leq k < m$ and $1 \leq j,j'\leq r$.  We claim that if $A'_{1,\mb{i}}$ is regular, so is $A'_{1,\mb{i}'}$.  This is straightforward unless $j = j'$ and $C_{ji} \neq 0$, so we restrict our attention to this case.  The argument in each of the above cases is essentially the same, so we will only consider Case 1 in detail.

Let $P_1$ and $P_2$ ($P'_1$ and $P'_2$) be the two monomials appearing in the right-hand side of the exchange relation defining $A'_{1,\mb{i}}$ ($A'_{1,\mb{i}'}$).  We must show that ${\De^{\om_i}_{e,s_i}}$ divides $P'_1 + P'_2$ in $\BC[G^{u,v}]$ given that it divides $P_1 + P_2$. 

If $u' = u_{\leq k}$, $v' = v_{>k}$, one can check that
\[
P'_1 + P'_2 = \frac{\left( P_1(\De^{\om_j}_{u',v's_j}\De^{\om_j}_{u's_j,v'})^{-C_{ji}} + P_2 (\De^{\om_j}_{u',v'}\De^{\om_j}_{u's_j,v's_j})^{-C_{ji}}  \right)}{((\De^{\om_j}_{u',v's_j})^{[k^- \nin E_+]}(\De^{\om_j}_{u's_j,v'})^{[k^++ \in E_+]} \De^{\om_j}_{u',v'})^{-C_{ji}}}.
\]
Here, e.g., $[k^- \in E_+]$ is the function which is 1 if $k^- \in E_+$, and 0 otherwise.  By \cref{prop:prime}, ${\De^{\om_i}_{e,s_i}}$ and the denominator of the right-hand side are relatively prime, so it suffices to show that ${\De^{\om_i}_{e,s_i}}$ divides the numerator.  This in turn is equivalent to showing that ${\De^{\om_i}_{e,s_i}}$ divides 
\[
(\De^{\om_j}_{u',v's_j}\De^{\om_j}_{u's_j,v'})^{-C_{ji}} - (\De^{\om_j}_{u',v'}\De^{\om_j}_{u's_j,v's_j})^{-C_{ji}},
\]
or simply that it divides 
\[
\De^{\om_j}_{u',v's_j}\De^{\om_j}_{u's_j,v'} - \De^{\om_j}_{u',v'}\De^{\om_j}_{u's_j,v's_j}.
\]
But since ${\De^{\om_i}_{e,s_i}} = {\De^{\om_i}_{u',v'}}$, this follows from \cref{prop:gendetid}.
\end{proof}

\begin{lemma}\label{lem:mutiso}
There is an open immersion $a_{\Si_{\mb{i}}}: \CA_{\Si_{\mb{i}}} \to G^{u,v}$ such that the generalized minors $A_i$ from \cref{defn:Aminors} pull back to the corresponding cluster variables on $\CA_{\Si_{\mb{i}}}$.  If $k \in I \setminus I_0$ is any unfrozen index and $\CA_k := \CA_{\mu_k(\Si_{\mb{i}})}$, then there is also an open immersion $a_k: \CA_k \to G^{u,v}$ forming a commutative diagram

%\centerline{\includegraphics{CEKMimage4.pdf}}

\vspace{-2mm}
\[
\begin{tikzcd}
\CA_{\Si_{\mb{i}}} \arrow[dashed]{rr}{\mu_k} \arrow{dr}[swap]{a_{\Si_{\mb{i}}}} && \CA_k \arrow{dl}{a_k} \\
& G^{u,v}. & 
\end{tikzcd}
\]
\noindent In particular, the regular functions $\{A_i |  i \in I, i \neq k \} \cup \{A'_k\} \subset \BC[G^{u,v}]$ pull back to the corresponding cluster variables on $\CA_k$.
\end{lemma}
\begin{proof}
The existence of the stated map $a_{\Si_{\mb{i}}}$ follows readily from \cref{prop:factorization,thm:main}.  Moreover, $a_{\Si_{\mb{i}}}$ is birational, hence there is a unique rational map $a_k$ making the given diagram commute; we claim it is in fact regular.  

There is a commutative square

%\centerline{\includegraphics{CEKMimage5.pdf}}

\vspace{-2mm}
\[
\begin{tikzcd}
\CA_k \arrow[dashed]{r}{a_k} \arrow[two heads]{d}{p'_M} & G^{u,v} \arrow[two heads]{d}{p_G} \\
\CX_k \arrow{r}{x_k} & G_{\Ad}^{u,v},
\end{tikzcd}
\]
\noindent where $x_k$ is the regular map defined in \cref{prop:x_k}.  Since $a_k$ is birational and the remaining maps are regular and dominant, the diagram embeds $\BC[\CX_k]$ and $\BC[G_{\Ad}^{u,v}]$ as subalgebras of the function field $\BC(\CA_k)$.  Moreover, we have $\BC[G_{\Ad}^{u,v}] \subset \BC[\CX_k]$ inside $\BC(\CA_k)$.  

Since $p'_M$ is finite and $\CA_k$ is normal, $\BC[\CA_k]$ is the integral closure of $\BC[\CX_k]$ in $\BC(\CA_k)$.  For the same reason, $\BC[G^{u,v}]$ is the integral closure of $\BC[G_{\Ad}^{u,v}]$ in $\BC(\CA_k)$.  But then the containment $\BC[G_{\Ad}^{u,v}] \subset \BC[\CX_k]$ inside $\BC(\CA_k)$ implies a containment $\BC[G^{u,v}] \subset \BC[\CA_k]$ of their integral closures, and it follows that $a_k$ is regular.  

It is clear from the construction that $a_k$ pulls back the regular functions $\{A_i | i \in I, i \neq k\} \cup \{A'_k\}$ on $G^{u,v}$ to the corresponding cluster variables on $\CA_k$.  It follows in particular that $a_k$ is injective.  But an injective birational morphism of smooth varieties is an open immersion, and the proposition follows.  
\end{proof}

\begin{lemma} \label{lem:codim}
Let $U \subset G^{u,v}$ be the open subset
\[
U := \CA_{\Si_{\mb{i}}} \cup \bigcup_{k \in I \setminus I_0} \CA_k,
\]
where we identify $\CA_{\Si_{\mb{i}}}$, $\CA_k:= \CA_{\mu_k(\Si_{\mb{i}})}$ with their images in $G^{u,v}$ following \cref{lem:mutiso}.  Then the complement of $U$ in $G^{u,v}$ has complex codimension greater than 1.
\end{lemma}
\begin{proof}
We first claim that the unfrozen generalized minors $A_k$ are distinct irreducible elements of $\BC[G^{u,v}]$, while the frozen ones are units.  If $k$ is frozen, either $k < 0$ or $k^+ = m+1$.  In the former case, $A_k = \De_{e,v^{-1}}^{\om_{|i_k|}}$, while in the latter $A_k = \De_{u,e}^{\om_{|i_k|}}$.  But in either case the fact that $A_k$ is nonvanishing on $G^{u,v}$ follows easily from the definition of the generalized minors.

Observe then that a Laurent monomial $M = \prod_{k \in I}A_k^{n_k}$ in the initial cluster variables is regular on $G^{u,v}$ if and only if $n_k \geq 0$ for all unfrozen $k$.  This follows from the definition of $A'_k$, since $M$ is regular on $\CA_k$ and hence expressible as a Laurent polynomial in $A'_k$ and the $A_i$ with $i \neq k$.  Suppose then that for some unfrozen index $k$ we can write $A_k$ as a product of two regular functions $P,Q \in \BC[G^{u,v}]$.  Clearly $P$ and $Q$ are themselves Laurent monomials in the $A_i$.  But since $PQ = A_k$, one of them must only involve frozen variables, hence is a unit in $\BC[G^{u,v}]$.  The fact that they are distinct is clear since their restrictions to $\CA_{\Si_{\mb{i}}}$ are distinct.  

We now claim that each $A'_k$ is the product of some irreducible element $A''_k \in \BC[G^{u,v}]$ and a Laurent monomial in the $A_i$ with $i \neq k$.  For suppose $P$ is an irreducible factor of $A'_k$.  Then $P$ must be expressible as a Laurent monomial in $A'_k$ and the $A_i$ with $i \neq k$, since it divides $A'_k$.  On the other hand, since $P$ is regular on $\CA_{\Si_{\mb{i}}}$, it follows from the definition of $A'_k$ that $A'_k$ appears with a nonnegative exponent in this monomial expression.  But then in the prime factorization of $A'_k$ there is exactly one irreducible factor such that this exponent is 1, and the statement follows.  Again, it is clear that this irreducible element $A''_k$ is distinct from the $A_i$ since their restrictions to $\CA_k$ are distinct.

Finally, we observe that the complement $G^{u,v} \setminus U$ is the locus where either $A_j$ and $A_k$ vanish for two distinct $j,k \in I$, or $A''_k$ and $A_k$ vanish for some $k \in I\setminus I_0$.  Let $x \in G^{u,v}$ be any element in the complement of $U$.  Since $x \nin \CA_{\Si_{\mb{i}}}$, $A_k(x)$ must equal zero for some $k \in I\setminus I_0$.  But $x \nin \CA_k$, so either $A''_k(x) = 0$ or $A_j(x) = 0$ for some $j \neq k$.  Thus $G^{u,v} \setminus U$ is the union of finitely many subvarieties cut out by two distinct irreducible equations, and the lemma follows.
\end{proof}

\begin{thm}\label{prop:mainX}
There is a regular map $x_{|\Si_{\mb{i}}|}: \CX_{|\Si_{\mb{i}}|} \to G_{\Ad}^{u,v}$ extending the map $\CX_{\Si_{\mb{i}}} \to G_{\Ad}^{u,v}$ of \cref{defn:Xdef}.    We have a commutative diagram

%\centerline{\includegraphics{CEKMimage6.pdf}}

\vspace{-2mm}
\[
\begin{tikzcd}
\CA_{|\Si_{\mb{i}}|} \arrow{r}{a_{|\Si_{\mb{i}}|}} \arrow{d}{p_M} & G^{u,v} \arrow{d}{p_G} \\
\CX_{|\Si_{\mb{i}}|} \arrow{r}{x_{|\Si_{\mb{i}}|}} & G_{\Ad}^{u,v},
\end{tikzcd}
\]
\noindent where $p_M$ and $p_G$ are as defined in \cref{thm:ensthm}
\end{thm}
\begin{proof} It follows from \cref{prop:XtoA} that $p_M$ is well-defined and that there is a rational map $x_{|\Si_{\mb{i}}|}$ making the diagram commute.  Let $\Si'$ be any seed mutation equivalent to $\Si_{\mb{i}}$ and let $x'$ be the restriction of this rational map to $\CX_{\Si'}$; it will follow that $x_{|\Si_{\mb{i}}|}$ is regular if we show that each such $x'$ is regular. 

We have a commutative diagram 

%\centerline{\includegraphics{CEKMimage7.pdf}}

\vspace{-2mm}
\[
\begin{tikzcd}
\CA_{\Si'} \arrow{r}{a'} \arrow[two heads]{d}{p'_M} & G^{u,v} \arrow[two heads]{d}{p_G} \\
\CX_{\Si'} \arrow[dashed]{r}{x'} & G_{\Ad}^{u,v},
\end{tikzcd}
\]
\noindent where $a'$ is the restriction of $a_{|\Si_{\mb{i}}|}$ to $\CA_{\Si'}$.  If we pull back $\BC[G^{u,v}_{\Ad}]$ along $x' \circ p'_M$ to the function field $\BC(\CA_{\Si'})$, we see that its image is contained in $\BC(\CX_{\Si'})$.  On the other hand,  if we perform the same pullback along $p_G \circ a'$, we see that the image of $\BC[G^{u,v}_{\Ad}]$ is contained in $\BC[\CA_{\Si'}]$.  Since $p'_M$ is surjective, any rational function on $\CX_{\Si'}$ which pulls back to a regular function on $\CA_{\Si'}$ must have been regular on $\CX_{\Si'}$.  Thus the intersection of $\BC(\CX_{\Si'})$ and $\BC[\CA_{\Si'}]$ in $\BC(\CA_{\Si'})$ is exactly $\BC[\CX_{\Si'}]$.  Thus $x'$ pulls back $\BC[G^{u,v}_{\Ad}]$ to $\BC[\CX_{\Si'}]$, hence is regular.
\end{proof}

\subsection{Poisson Brackets of $\CX$-coordinates}\label{sec:Poisson}

We now complete the proof of \cref{thm:ensthm}, demonstrating that the map $x_{|\Si_{\mb{i}}|}: \CX_{|\Si_{\mb{i}}|} \to G^{u,v}_{\Ad}$ is Poisson.  First we recall some rudiments of Poisson-Lie theory \cite{Chari1995}.

Any symmetrizable Kac-Moody group $G$ is a Poisson ind-algebraic group in a canonical way \cite{Williams2012}.  That is, its coordinate ring is equipped with a continuous Poisson bracket such that the multiplication map $G \times G \to G$ is Poisson.  The double Bruhat cells of $G$ are Poisson subvarieties, and on any given double Bruhat cell $H$ acts transitively on the set of symplectic leaves by left multiplication.  This standard Poisson structure is characterized by the fact that the maps
\[
\varphi_i:SL_2^{d_i} \to G
\]
are Poisson.  Here $SL_2^{d_i}$ refers to the following Poisson-Lie structure on $SL_2$: if we write
\[
SL_2 = \left\{ \begin{pmatrix} A & B \\ C & D \end{pmatrix} : AD-BC =1 \right\},
\]
then the brackets of the coordinate functions on $SL_2^{d_i}$ are given by
\begin{gather*}
\{B,A\} = \frac{d_i}{2} AB, \quad \{B,D\} = -\frac{d_i}{2}BD, \quad \{B,C\} = 0, \\
\{C,A\} = \frac{d_i}{2}AC, \quad \{C,D\} = -\frac{d_i}{2}CD, \quad \{D,A\} = d_i BC.
\end{gather*}
The Cartan subgroup of $G$ is a Poisson-Lie subgroup endowed with the trivial Poisson structure.  Then since the kernel of $G \to G_{\Ad}$ is a discrete subgroup of $H$, $G_{\Ad}$ in turn inherits the standard Poisson structure from $G$.

\begin{thm}\label{prop:Poissonmap}
The regular map $x_{|\Si_{\mb{i}}|}: \CX_{|\Si_{\mb{i}}|} \to G_{\Ad}^{u,v}$ defined in \cref{prop:mainX} is Poisson.\footnote{In finite type this is the result of \cite[Proposition 3.11]{Fock2006}.}
\end{thm}
\begin{proof}
Since $\CX_{\Si_{\mb{i}}}$ is dense in $\CX_{|\Si_{\mb{i}}|}$, it suffices to check that the original map $\CX_{\Si_{\mb{i}}} \to G_{\Ad}^{u,v}$ is Poisson.  Thus if $\{,\}_G$ denotes the restriction of the standard Poisson bracket on $G_{\Ad}^{u,v}$, we must check that
\[
\{X_j,X_k\}_G = b_{jk}d_kX_j X_k
\]
for all $j,k \in I$.  We recall that the upper and lower Borel subgroups of $SL_2^{d}$ are Poisson subgroups.  For $1 \leq k \leq m$ let $B_{i_k}$ denote the positive Borel subgroup of $SL_2^{d_{|i_k|}}$ if $\ep_k = 1$, and its negative Borel subgroup if $\ep_k = -1$.  There is then a Poisson map
\[
m_{\mb{i}}: H \times B_{i_1} \times \cdots \times B_{i_m} \to G^{u,v}_{\Ad}
\]
given by the maps $\varphi_{|i_k|}$ and multiplication in $G_{\Ad}$, and whose image coincides with $\CX_{\Si_{\mb{i}}}$.  We define coordinates $P_k,Q_k$ on each $B_{i_k}$ by
\[
B_{i_k} = 
\left\{ \begin{pmatrix} P_k & Q_k \\ 0 & P_k^{-1} \end{pmatrix}:  (P_k,Q_k) \in \BC^* \times \BC \right\}
\]
for $\ep_k = +1$ and
\[
B_{i_k} = \left\{ \begin{pmatrix} P_k & 0 \\ Q_k & P_k^{-1} \end{pmatrix}:  (P_k,Q_k) \in \BC^* \times \BC \right\}
\]
for $\ep_k = -1$.  In either case the Poisson bracket on $H \times B_{i_1} \times \cdots \times B_{i_m}$ is given by
\[
\{P_j,Q_k\} = \frac{d_{|i_k|}}{2}P_kQ_k\de_{jk}.
\]

Since $m_{\mb{i}}$ is dominant and Poisson, the brackets among the $X_i$ are determined by the brackets of their pullbacks along $m_{\mb{i}}$.  Moreover, since the coordinate functions on $H$ are Casimirs, it suffices to consider the restrictions of these pullbacks to $B_{i_1} \times \cdots \times B_{i_m}$.  

Note that
\begin{align*}
\varphi_{|i_k|}(B_{i_k}) & = P^{\al_{|i_k|}^\vee}_k(P_k^{-1}Q_k^{\ep_k})^{\om_{|i_k|}^\vee}E_{i_k}(P_kQ_k^{-\ep_k})^{\om_{|i_k|}^\vee}\\
& = \biggl( \prod_{\substack{j \neq |i_k| \\ 1 \leq j \leq \wt{r}}} P_k^{C_{|i_k|,|i_j|}\om_{j}^\vee} \biggr) (P_kQ_k^{\ep_k}))^{\om_{|i_k|}^\vee}E_{i_k}(P_kQ_k^{-\ep_k})^{\om_{|i_k|}^\vee}.
\end{align*}
Then writing out $m_{\mb{i}}$ explicitly and comparing with \cref{defn:Xdef} one obtains
\[
m_{\mb{i}}^*X_j = (P_jQ_j^{-\ep_j})^{[j>0]}(P_{j^+}Q_{j^+}^{\ep_{j^+}})^{[j^+ \leq m]}\biggl( \prod_{\substack{j<k<j^+ \\ k>0}}P_k^{C_{|i_k|,|i_j|}} \biggr).
\]
But now one can check directly that
\begin{align*}
\frac{\{X_j,X_k\}_G}{X_jX_k} & = \ep_j d_k [j = k^+] - \ep_k d_k[j^+ = k] + \ep_j d_j \frac{C_{kj}}{2}[k<j<k^+][j>0] \\
& \quad - \ep_{j^+} d_j \frac{C_{kj}}{2}[k<j^+<k^+][j^+\leq m] - \ep_k d_k \frac{C_{kj}}{2}[j<k<j^+][k>0] \\
& \quad + \ep_{k^+} d_k \frac{C_{kj}}{2}[j<k^+<j^+][k^+\leq m]\\
& = b_{jk}d_k.
\end{align*}
\end{proof}

\bibliography{CEKMRevised} 
\bibliographystyle{alpha}

\end{document}